\documentclass[11pt]{amsart}
\input epsf
\usepackage{latexsym}
\usepackage{amssymb,amsmath,amsthm,amscd, amssymb,
graphicx, enumerate, verbatim}
\usepackage[all]{xy}

\numberwithin{equation}{section}
\newtheorem{cor}[equation]{Corollary}
\newtheorem{lem}[equation]{Lemma}

\newtheorem{thm}[equation]{Theorem}

\newtheorem{Example}[equation]{Example}
\newenvironment{ex}{\begin{Example}\rm}{\end{Example}}
\newtheorem{remark}[equation]{Remark}
\newenvironment{rmk}{\begin{remark}\rm}{\end{remark}}
\def\co{\colon\thinspace}

\newcommand{\Int}{\mbox{Int}}
\newcommand{\Diff}{\mbox{Diff}}
\newcommand{\Conf}{\mbox{Conf}}

\newcommand{\id}{\mbox{\bf id}}
\newcommand{\e}{\varepsilon}

\def\bu{\bullet}
\def\a{\alpha}
\def\G{\Gamma}
\def\de{\delta}
\def\g{\gamma}
\def\o{\omega}

\def\Om{\Omega}
\def\Si{\Sigma}
\def\O{\mathcal O}
\def\X{\mathcal X}
\def\M{\mathcal M}

\def\b{\beta}

\def\d{\partial}
\def\k{\kappa}
\def\th{\theta}
\def\r{\rho}
\def\o{\omega}

\def\s{\sigma}
\def\l{\lambda}

\def\Z{\mathbb{Z}}
\def\N{\mathbb{N}}
\def\S1{\bf S^1}

\newcommand{\R}{{\mathbb R}}
\newcommand{\C}{{\mathbb C}}

\makeatletter
\def\equalsfill{$\m@th\mathord=\mkern-7mu
\cleaders\hbox{$\!\mathord=\!$}\hfill
\mkern-7mu\mathord=$}
\makeatother

\begin{document}

\abovedisplayskip=6pt plus3pt minus3pt
\belowdisplayskip=6pt plus3pt minus3pt

\title[Spaces of nonnegatively curved surfaces]
{\bf Spaces of nonnegatively curved surfaces}

\rm
%53C21 Methods of Riemannian geometry, including PDE methods; curvature restrictions
%57N20 Topology of infinite-dimensional manifolds

\author{Taras Banakh}

\address{T. Banakh\\ Ivan Franko National University of Lviv (Ukraine) and Jan Kochanowski University in Kielce (Poland)} \email{t.o.banakh@gmail.com}

\author{Igor Belegradek}

\address{Igor Belegradek\\School of Mathematics\\ Georgia Institute of
Technology\\ Atlanta, GA 30332-0160}\email{ib@math.gatech.edu}

%\thanks{}

\date{}
\begin{abstract} 
We determine the homeomorphism type of the space of smooth 
complete nonnegatively curved metrics on $S^2$, $RP^2$, and $\C$ 
equipped with the topology of $C^\g$ uniform convergence on compact sets,
when $\g$ is infinite or is not an integer. If $\g=\infty$, the space of metrics is homeomorphic
to the separable Hilbert space. If $\g$ is finite and not an integer,
the space of metrics is homeomorphic to the countable power of the linear span of the Hilbert cube.
We also prove similar results for some other spaces of metrics including
the space of complete smooth Riemannian metrics on an arbitrary manifold.
\end{abstract}
\maketitle
%\tableofcontents

\section{Introduction}

In this paper {\em smooth\,} 
means $C^\infty$, by a {\em manifold\,}
we mean a smooth finite-dimensional manifold without boundary,
the {\it $C^\g$ topology\,} on a set of maps between two fixed manifolds 
is the topology of uniform $C^{\g}$ convergence on compact smooth domains, where $0\le \g\le\infty$,
see Section~\ref{sec: Cgamma}. 
The $C^\g$ topology is called {\it H\"older\,} if $\g$ is finite and not an integer. 
Note that the $C^\g$ topology on any 
set of smooth maps is metrizable and separable.

Given a connected manifold $M$ of positive dimension
let $\mathcal R^{\g}(M)$ denote the space of all complete smooth Riemannian metrics
on $M$ endowed with the $C^\g$ topology. For finite $\g$ 
the mismatch between infinite smoothness of the metrics 
and their $C^\g$ convergence results in a larger number of 
converging sequences. 
This situation arises naturally in applications, e.g., according to~\cite{AndChe}  
if a sequence of smooth metrics $g_i$ on a closed manifold 
has uniform lower bounds on injectivity radius and Ricci curvature as well as an upper diameter bound, 
then for some diffeomorphisms $\phi_i$ the family  of pullback metrics $\phi_i^*g_i$
is $C^\g$ precompact for every $\g<1$,
see~\cite{Pet-conv-msri} for other results of this type.
 
Our first result  recognizes the homeomorphism type of $\mathcal R^{\g}(M)$:

\begin{thm} 
\label{thm: all metrics} If $M$ is a connected manifold of positive dimension, then
$\mathcal R^{\infty}(M)$ is homeomorphic to
$\ell^2$, and  $\mathcal R^{\g}(M)$ is homeomorphic to
$\Sigma^\o$ for every finite $\g$.
\end{thm}
Here $\ell^2$ is the separable Hilbert space, 
$\Si$ is the linear span of the standard Hilbert cube in $\ell^2$, and 
$\Si^\o$ is the product of countably many copies of $\Si$.
Since the space $\Si^\o$ may be unfamiliar to the reader,
let us mention that it is a locally convex linear space which 
is a countable union of nowhere dense sets, so
$\Si^\o$ is not completely metrizable.
If $\Om$ is either $\ell^2$ or $\Si^\o$, then it has the following properties, 
see Section~\ref{sec: absorb}:
\begin{itemize}
\item[(a)] $\Om$ is not $\s$-compact, and in particular, not locally compact.
\item[(b)] Any two open homotopy equivalent subsets of $\Om$ are homeomorphic. 
\item[(c)] The complement to any compact subset of $\Om$ is contractible.
\item[(d)] Any homeomorphism of two compact subsets of $\Om$ extends to a homeomorphism of $\Om$.
\end{itemize}
In (c) and (d) one can replace the phrase ``compact subset'' by ``$Z$-set'';
many examples of $Z$-sets can be found in Section~\ref{sec: absorb}.

A starting point in the proof of Theorem~\ref{thm: all metrics} is the
convexity of $\mathcal R^{\g}(M)$ in the locally convex linear space 
of all smooth symmetric $2$-tensors on $M$, with the $C^\g$ topology.  
As a caution we note that there is no meaningful
homeomorphism classification of convex subsets of locally convex linear spaces, 
e.g., the union of an open ball in $\ell^2$ with any subset of its boundary sphere
is convex. Such pathologies do not occur in the case at hand. 
There is a standard machinery that makes
recognizing the homeomorphism type of $\mathcal R^{\infty}(M)$ quite easy.
The case of finite $\g$ is more delicate and it exploits a certain argument 
in~\cite{Ban-sigma-to-the-omega} together with a substantial amount
of established techniques.

For a constant $\l$ let  $\mathcal R_{\ge \l}^{\g}(M)$
denote the subspace of $\mathcal R^{\g}(M)$ consisting of the metrics
with sectional curvature $\ge \l$; the space $\mathcal R_{>\l}^{\g}(M)$
is defined similarly. The subspaces are generally non-convex in $\mathcal R^{\g}(M)$, 
and little is known about their topological properties. A natural idea is to find a
homeomorphic parametrization of $\mathcal R_{\ge \l}^{\g}(M)$ by a space to which
techniques of infinite dimensional topology apply.
In this paper we do that when $M$ is $\C$, $S^2$ or $RP^2$, the surfaces
admitting a complete non-flat metric of nonnegative curvature, provided $\g\notin\Z$ 
and also $\l=0$ for $M=\C$. 
Here are our main results: 

\begin{thm} 
\label{thm: main compact}
If $M$ is $S^2$ or $RP^2$, then for every real $\l$ the spaces $\mathcal R_{\ge \l}^{\g}(M)$ 
and $\mathcal R_{>\l}^{\g}(M)$ are homeomorphic to
\begin{itemize}\vspace{-4pt}
\item[\textup{(1)}] 
$\ell^2$ if $\g=\infty$,\vspace{2pt}
\item[\textup{(2)}] 
$\Sigma^\o$ if $\g$ is finite and not an integer.
\end{itemize}
\end{thm}

\begin{thm} 
\label{thm: main C}
The spaces $\mathcal R_{\ge 0}^{\g}(\C)$ 
and $\mathcal R_{>0}^{\g}(\C)$ are homeomorphic to
\begin{itemize}\vspace{-4pt}
\item[\textup{(1)}] 
$\ell^2$ if $\g=\infty$,\vspace{2pt}
\item[\textup{(2)}] 
$\Sigma^\o$ if $\g$ is finite and not an integer.
\end{itemize}
\end{thm}

The homeomorphism of $\mathcal R_{\ge 0}^{\infty}(\C)$ and $\ell^2$ was already established in~\cite{BelHu-modr2}.

To explain the assumption $\g\notin\Z$, let us sketch the analytic ingredients 
of Theorem~\ref{thm: main C} in the case of $\mathcal R_{\ge 0}^{\g}(\C)$.
It is well-known, see~\cite{BlaFia}, that 
any complete metric of nonnegative curvature on $\C$ is conformally equivalent 
to the standard Euclidean metric $g_{_0}$. Hence it can be written as $\phi^*e^{-2u}g_{_0}$ where
$\phi$ is an orientation-preserving diffeomorphism of $\C$ and $u$ is a smooth function on $\C$.
By composing $\phi$ with a conformal automorphism of $\C$ one can assume that $\phi$ fixes
a pair of points, say $0$ and $1$, which determines $\phi$ uniquely. 
Such $\phi$'s form a closed contractible subgroup  of the diffeomorphism group of $\C$, which
we denote $D^{\g+1}(\C)$; the superscript $\g+1$ indicates that the group
is equipped with the $C^{\g+1}$ topology, which is natural because $\phi^*e^{-2u}g_{_0}$
involves the differential of $\phi$. Nonnegativity of the curvature
is equivalent to subharmonicity of $u$. Completeness of $e^{-2u}g_{_0}$
imposes further restrictions on $u$, making it roughly speaking of sublogorithmic growth,
and such $u$'s form a convex subset $\O_{\ge 0}^\g(\C)$ in the space of  smooth functions
equipped with the $C^\g$ topology, 
see~\cite{BelHu-modr2}. With the above restrictions on $\phi$ and $u$ 
the map $(\phi, u)\to \phi^*e^{-2u}g_{_0}$ becomes a continuous bijection.
Moreover, smooth dependence of solutions of Beltrami equation on the dilatation
shows that the continuous bijection is a homeomorphism for $\g\notin\Z$, see~\cite{BelHu-modr2, BelHu-modr2-err};
we expect this to fail when $\g$ is an integer. In summary, for $\g\notin\Z$
the space $\mathcal R_{\ge 0}^{\g}(\C)$ is homeomorphic to the product $D^{\g+1}(\C)\times\O_{\ge 0}^\g(\C)$
where $D^{\g+1}(\C)$ is a certain diffeomorphism group and $\O_{\ge 0}^\g(\C)$ is some convex set 
of smooth functions on $\C$. The homeomorphism
type of the factors $D^{\g+1}(\C)$, $\O_{\ge 0}^\g(\C)$ can be determined along the lines of 
the proof of Theorem~\ref{thm: all metrics}, and the conclusion is that 
$D^{\g+1}(\C)$, $\O_{\ge 0}^\g(\C)$ 
are each homeomorphic to $\ell^2$ or $\Si^\o$ depending on whether $\g$ is infinite or finite.
The case $\mathcal R_{>0}^{\g}(\C)$ is similar, and 
the proof of Theorem~\ref{thm: main compact} follows the same outline except 
that~\cite{BelHu-modr2} is not needed. 

The above proof requires that all metrics lie in the same 
conformal class, and in particular, the proof does not extend 
to $\mathcal R_{\ge \l}^{\g}(\C)$ with $\l<0$
or to $\mathcal R_{\ge\l}^{\g}(M)$ where $M$ is a closed surface of 
nonpositive Euler characteristic. 
The space $\mathcal R_{\ge \l}^{\g}(\C)$ with $\l>0$ is empty
due to the Myers theorem: Any complete Riemannian 
manifold of curvature $\ge \l>0$ is compact. 
Another essential feature of the proof
is the convexity of $\O_{\ge 0}^\g(\C)$ which prevents us from
treating spaces of metrics with
two sided curvature bounds such as the subspace of 
$\mathcal R_{\ge 0}^{\g}(\C)$ consisting of metrics with curvature $\le 1$.

Techniques of this paper apply to other convex sets of smooth maps not necessarily equipped
with the $C^\g$ topology, and whenever
possible the results are stated so that they can be easily used in other contexts,
see Sections~\ref{sec: aux function space}--\ref{sec: M2-univ operator D}.

{\bf Structure of the paper}
Section~\ref{sec: Cgamma} is a review of definitions and properties of the $C^\g$ topology.
The necessary infinite dimensional topology background is collected in Section~\ref{sec: absorb}.
Sections~\ref{sec: aux function space} and~\ref{sec: M2-univ operator D}
contain the main technical results needed to treat the case when $\g$ is finite. 
Theorem~\ref{thm: all metrics} is proved in
Sections~\ref{sec: R smooth} and~\ref{sec: R C^gamma}.
Basic information on the spaces of metrics on surfaces is collected
in Sections~\ref{sec: unif M} and~\ref{sec: unif of C}.
Finally, Theorems~\ref{thm: main compact}--\ref{thm: main C} are proved in
Sections~\ref{sec: O smooth}--\ref{sec: O absorb}.

{\bf Acknowledgments}
We are thankful to the referee for editorial suggestions. 
Belegradek is grateful for NSF support: DMS-1105045. Some of the ideas in the present paper
go back to Belegradek's joint work with Jing Hu, which appeared in her thesis~\cite{Hu-thesis}.
The results we prove here are much stronger, e.g., one of the main theorems in~\cite{Hu-thesis}
is the topological homogeneity of $\mathcal R_{\ge 0}^\g(S^2)$ when $\g\ge 2$ and $\g\not\in \Z$,
while we determine the homeomorphism type of $\mathcal R_{\ge 0}^\g(S^2)$ for $\g\not\in \Z$.

\section{A review of $C^\g$ topology}
\label{sec: Cgamma}

In this section we recall definitions and basic properties of the $C^\g$
topology.
We stress that {\em smooth\,} always means $C^\infty$.
Fix $\g\in [0,\infty]$, and if $\g$ is finite write it uniquely
as $\g=k+\a$ where $k$ is a nonnegative integer and $\a\in [0,1)$.

If $\a=0$, then the linear space $C^k(D, \R^n)$ of $k$-times continuously differentiable 
maps from a smooth compact domain $D\subset\R^m$ to $\R^n$ is Banach with respect
to the norm 
\begin{equation}
\|f\|_{_{C^k(D,\R^n)}}=\max\{|f^{(p)}(x)|\co |p|\le k,\, x\in D\}
\end{equation}
where $f^{(p)}$ is a partial derivative for the 
multi-index $p$ of order $|p|$. 
If $\a\in (0,1)$, we say that 
a map $f\in C^0(D,\R^n)$ is {\em $\a$-H\"older\,} if the following
quantity is finite
\[
\|f\|_{_{C^\a(D,\R^n)}}=\sup\left\{\frac{|f(x)-f(y)|}{|x-y|^\a}\co x,y\in D,\ x\neq y\right\}
\]
where $|\cdot|$ is the Euclidean norm.
The maps in $C^k(D, \R^n)$ whose $k$-th partial derivatives are all $\a$-H\"older
form a linear subspace, denoted $C^{k+\a}(D; \R^n)$, which is Banach 
with respect to the norm
\[
\|f\|_{_{C^{k+\a}(D,\R^n)}}=\|f\|_{_{C^k(D,\R^n)}}+\sup_{|\mu|=k}\|D^\mu f\|_{_{C^\a(D,\R^n)}}.
\] 

A map of smooth manifolds 
is called $C^{k+\a}$ if whenever it is written in local coordinates
as a map between open sets of Euclidean spaces
its restriction to any compact smooth domain 
has the property that all $k$-th partial derivatives are $\a$-H\"older.
With this definition the elements of $C^{k+\a}(D, \R^n)$ are precisely
the $C^{k+\a}$ maps from $D$ to $\R^n$ because if 
$D^\prime$ is another compact domain in $\R^m$, then the composite
of a map in $C^{\infty}(D^\prime, D)$  with a map in $C^{k+\a}(D, \R^n)$ 
is in  $C^{k+\a}(D^\prime, \R^n)$, see~\cite[Section 2.2]{BHS-holder}.

Clearly any $C^{k+\a}$ map is $C^k$. Less trivially, for every
$0<\g<\b\le \infty$ any $C^{\b}$ map is $C^\g$,
see~\cite[Section 6.8]{GilTru} for the case when $\g$, $\b$
are not integers. The product of two real valued $C^{k+\a}$ functions is $C^{k+\a}$,
see~\cite[Section 4.1]{GilTru}.
If $k\ge 1$, then the composition of $C^{k+\a}$ maps is $C^{k+\a}$,
and the inverse of an invertible $C^{k+\a}$ map is $C^{k+\a}$, 
see~\cite[Sections 2.1--2.2]{BHS-holder}.

We are now ready to define the $C^\g$ topology on the set $C^\g(M,N)$
of $C^\g$ maps between smooth manifolds $M$, $N$, cf.~\cite[Chapters 7--8]{Pal-glob-anal}.
Fix a smooth embedding $i_{_{N}}$ of $N$ as a closed submanifold of some $\R^n$.
Postcomposing with $i_{_{N}}$ defines an inclusion of $C^\g(M,N)$
into $C^\g(M,\R^n)$. We shall define the $C^\g$ topology on
$C^\g(M,\R^n)$ and then give $C^\g(M,N)$ the subspace topology.
Cover $M$ by a countable family of smooth compact domains $D_j$ 
each lying in some coordinate chart. The Banach norm $\|\cdot\|_{_{C^{k+\a}(D_j,\R^n)}}$ 
on $C^\g(D_j,\R^n)$ defines a seminorm $p_j$ on $C^\g(M,\R^n)$ by restriction, i.e., 
$p_j(f)=\|f\vert_{_{D_j}}\|_{_{C^{k+\a}(D_j,\R^n)}}$. The countable family of seminorms
$p_j$ gives $C^\g(M,\R^n)$ a Fr{\'e}chet space structure,
and one can show that the structure is independent of the choices involved.

The space $C^\infty(M,\R^n)=\bigcap_{k\in\mathbb N} C^k(M, \R^n)$
gets the associated Fr{\'e}chet structure given by seminorms
$\|\cdot\|_{_{C^k(D_j,\R^n)}}$ where both $k$ and $j$ vary.
More details can be found in~\cite[Chapter 10, Example 1]{Tre-TVS-book} 
and~\cite[Sections 1 and 4 in Chapter 2]{Hir}.

Both $C^k(M,\R^n)$ and $C^\infty(M,\R^n)$ are separable, see e.g.~\cite[Lemma 3.2]{BelHu-modr2}, 
while it is well-known that 
$C^{k+\a}(M,\R^n)$ is non-separable if $\a\neq 0$, except when $\dim M$ or $n$ vanish.

\section{Infinite dimensional topology background}
\label{sec: absorb}

In this section any {\em space\,} is 
assumed metrizable and separable, and any {\it map\,} is assumed continuous;
we do not follow the convention in other sections.

By a {\em subspace\,} we always mean ``subset with subspace topology''; we never use
the term to mean a linear subspace.

A basic goal of infinite-dimensional topology is to give a homeomorphism classification
of naturally occurring spaces such as topological groups, linear spaces and their convex subsets.
A closely related problem is a characterization of $\Om$-manifolds for a suitable space $\Om$.
Here an {\em $\Om$-manifold\,} is a space in which every point 
has a neighborhood homeomorphic to an open subset of $\Om$.
Satisfactory answers are known, e.g., when $\Om$ is $\ell^2$ or 
$\Om$ is absorbing in the sense of Bestvina-Mogilski~\cite{BesMog}.

The exposition below follows~\cite{BP-book, BRZ-book, BesMog, DobMog-survey}.
The subject is quite technical so we attempt to limit the terminology to what
is relevant for our applications.

A {\it linear metric space\,} is a vector space with a translation-invariant metric.
A linear metric space is {\it locally convex\,} if its topology is given by a countable family of seminorms.
Any normed space is locally convex. A {\it Fr{\'e}chet space\,} is a complete locally convex
linear metric space. The {\it dimension of a convex subset $C$\,} is the dimension
of the vector space affinely isomorphic to the affine hull $\mathrm{Aff} C$ of $C$. 

Any convex set in a locally convex linear metric 
space is an AR (absolute retract), see~\cite[Corollary II.5.3]{BP-book}.
A space is an AR if and only if it is a contractible 
ANR (absolute neighborhood retract), see~\cite[Theorem II.5.1]{BP-book}.
A space that is locally an ANR is an ANR~\cite[Theorem II.5.1]{BP-book}.

A subset $B\subset X$ is {\it homotopy dense\,} if there is a homotopy
$h\co X\times I\to X$ with $h_0=\mathrm{id}$ and $h(X\times (0,1])\subset B$.
An embedding with a homotopy dense image is {\it homotopy dense}.
A subset is {\it homotopy negligible\,} if its complement is homotopy dense.
If $X$ is an ANR, then a subset $B$ of $X$ is homotopy dense 
if and only if each map $I^n\to X$ with $\d I^n\subset B$ can be uniformly approximated
rel boundary by maps $I^n\to B$, see~\cite[Theorem 1.2.2]{BRZ-book}. 
Other characterizations of homotopy dense subsets can be found in~\cite[Exercise 12 in 1.2]{BRZ-book}.

It follows that a subset $B$ is homotopy dense if and only if 
it is locally homotopy dense, see~\cite[Exercise 3 in 1.2]{BRZ-book}; the latter means that
each point has a neighborhood $U$ such that $U\cap B$ is homotopy dense in $U$.

Any homotopy dense subset of an $\ell^2$-manifold is infinite-dimensional~\cite[Exercise 6 in 1.3]{BRZ-book}.
Any homotopy dense subset of an ANR is an ANR, so in particular, any homotopy dense
subset of an $\ell^2$-manifold is an ANR.

Given an open cover $\mathcal U$ two maps $f,g\co Y\to X$ are {\it $\mathcal U$-close\,} 
if for every $y\in Y$ there is $U\in\mathcal U$ with $f(y), g(y)\in U$.
Let $Q$ denote the Hilbert cube.

A space $X$ has the {\em Strong Discrete Approximation Property (SDAP)\,}
if for every open cover $\mathcal U$ of $X$ each map $Q\times \N\to X$
is $\mathcal U$-close to a map $g\co Q\times \N\to X$ such that every point
of $X$ has a neighborhood that intersects at most one set of the family 
$\{g(Q\times\{n\})\}_{n\in\N}$. A space is homeomorphic
to a homotopy dense subset of an $\ell^2$-manifold if and only if 
it is an ANR with SDAP~\cite[Theorem 1.3.2]{BRZ-book}.  

The following is proved in~\cite[Propositions 5.2.1 and 5.2.6]{BRZ-book}: 

\begin{lem}
\label{lem: convex embeds as homot dense}
Let $C$ be a convex subset of a separable Fr{\'e}chet space. If the closure
$\bar C$ of $C$ is either not locally compact, or not contained in $\mathrm{Aff} C$, 
then $C$ is homeomorphic to a homotopy dense
subset of $\ell^2$. 
\end{lem}

A closed subset $B$ of a space $X$ is a {\it $Z$-set\,} if every map $Q\to X$
%(or equivalently, from $I^n$ for every $n\ge 0$~\cite[Proposition V.2.1]{BP-book})
can be uniformly approximated by a map whose range misses $B$. 
%The above definition is equivalent to assuming the same for each $I^n$, $n\ge 0$ 
%in place of $Q$, see~\cite[Proposition V.2.1]{BP-book}. 
A space is called $\s Z$ if it is a countable union of $Z$-sets.
If $X$ is an ANR, then a subset $B$ of $X$ is a $Z$-set if and only if $B$ is 
closed and homotopy negligible, see~\cite[Theorem 1.4.4]{BRZ-book} and~\cite[Proposition V.2.1]{BP-book}.

\begin{ex}
(1) The union of a locally finite family of $Z$-sets in an ANR is a $Z$-set~\cite[Exercise 1 in 1.4]{BRZ-book}.

(2) If $X$, $Y$ are ANR and $B$ is a $Z$-set in $X$, then $B\times Y$ is a $Z$-set in $X\times Y$. 

(3) If $X$ is homeomorphic to a homotopy dense subset of an $\ell^2$-manifold,
then any compact subset is a $Z$-set~\cite[Proposition 1.4.9]{BRZ-book}.
\end{ex}

There are some variations of the notion of a $Z$-set that are all equivalent
when the ambient space is a homotopy dense subset of an $\ell^2$-manifold,
see~\cite[Section 1.4]{BRZ-book} and~\cite[Proposition V.2.1]{BP-book}.
In particular, a {\it strong $Z$-set} in a space $X$
is a closed subset $A$ of $X$ such that for any open cover $\mathcal U$
there is a map $f\co X\to X$ that is $\mathcal U$-close to the identity
and such that $A$ is disjoint from the closure of $f(X)$.
Actually, we shall never use the definition because any strong $Z$-set is a $Z$-set,
while in a homotopy dense subset of an $\ell^2$-manifold any $Z$-set is strong, 
see ~\cite[Proposition 1.4.3]{BRZ-book}.
If an ANR is a countable union of strong $Z$-sets, then it
admits a homotopy dense embedding into an $\ell^2$-manifold, 
see~\cite[Proposition 1.9]{BesMog} and~\cite[Proposition 1.4.10]{BRZ-book}.
%In a space that admits a homotopy dense embedding into
%an $\ell^2$-manifold any compact subset is a  $Z$-set~\cite[Proposition 1.4.9]{BRZ-book}.

Every $Z$-set is nowhere dense (otherwise, its closure, which equals to the $Z$-set
itself, contains an open set $U$
and the constant map from the Hilbert cube to $U$ cannot be approximated by maps
whose range misses $U$). 

A map is a {\it $Z$-embedding\,} if its image is a $Z$-set. 
If $\mathcal C$ is a class of spaces, then a space $X$ is {\it strongly $\mathcal C$-universal\,} if 
for every open cover $\mathcal U$ of $X$, every $C\in\mathcal C$, every closed subset $B\subset C$, and every
map $f\co C\to X$ that restricts to a $Z$-embedding on $B$ there is a $Z$-embedding
$\bar f\co C\to X$ with $\bar f\vert_B=f\vert_B$  such that $f$, $\bar f$ are $\mathcal U$-close.
Note that if $X$ is strongly 
$\mathcal C$-universal, then any space in $\mathcal C$ is homeomorphic to a $Z$-set of $X$.

Given a class $\mathcal C$  of spaces, a space $X$ is {\it $\mathcal C$-absorbing\,}
if the following holds: \begin{itemize}
\item[\textup(i)] $X$ is strongly $\mathcal C$-universal,
\item[\textup(ii)] $X$ is the union of countably many $Z$-sets,
\item[\textup(iii)] $X$ is homeomorphic to a homotopy dense subset of an $\ell^2$-manifold,
\item[\textup(iv)] $X$ is the union of a countably many closed subsets homeomorphic to spaces in $\mathcal C$.
\end{itemize}

\begin{ex} 
Since $Z$-sets are nowhere dense, any $\mathcal C$-absorbing space is meager in itself.
Therefore by the Baire category theorem 
no $\mathcal C$-absorbing space is locally compact or completely metrizable
(no matter what $\mathcal C$ is).
In particular, $\ell^2$ is not $\mathcal C$-absorbing.
%No $\ell^2$-manifold is $\mathcal C$-absorbing 
%because (ii) fails, see~\cite[Theorem V.6.4]{BP-book}.  
\end{ex}

\begin{ex}
The linear span $\Sigma$ of the Hilbert cube in $\ell^2$ is $\M_0$-absorbing,
where $\M_0$ is the class of compact spaces, see~\cite[Theorem 6]{Mog-sigma-comp-1984}.
Note that $\Sigma$ is $\sigma$-compact.
\end{ex}

\begin{ex}
\label{ex: M2 and Sigma to the omega}
The space $\Si^\o$ is $\M_2$-absorbing, where
$\M_2$ be the class of spaces 
homeomorphic to $F_{\sigma\delta}$-sets in compacta, see~\cite[Exercise 3 in 2.4]{BRZ-book}. 
Note that $\Si^\o$ is not $\s$-compact (the product of infinitely many $\s$-compact noncompact
spaces is never $\s$-compact). 
\end{ex}

The term ``absorbing'' seem to have come from the following property:
If $\mathcal C$ is the class of compact spaces, and $X$
is a $\mathcal C$-absorbing space embedded into $\ell^2$,
then any compact subset of $\ell^2$ can be mapped into $X$ by a homeomorphism of $\ell^2$, 
see ~\cite[Theorem V.6.2]{BP-book}.

Absorbing spaces have remarkable properties
some of which are listed below (in a weakened form). 
Suppose $\mathcal C$  is a class of spaces that is closed under homeomorphisms, passing to
closed subsets, and closed-unions (the last property means that if $C$ is the union of two closed
subsets that belong to $\mathcal C$, then $C\in\mathcal C$). Suppose also that $\mathcal C$ is {\em local\,},
i.e., a space $X$ belongs to $\mathcal C$ if and only if each point of $X$ has a neighborhood that belongs to 
$\mathcal C$. 
Then the following holds:
\begin{itemize}
\item (Uniqueness) Any two homotopy equivalent $\mathcal C$-absorbing spaces
are homeomorphic, see~\cite[Theorem 3.1]{BesMog}.
\vspace{2pt}
\item ($Z$-set unknotting) If $A$, $B$ are $Z$-sets in a $\mathcal C$-absorbing
space $X$, then any homeomorphism $A\to B$ that is homotopic to the inclusion
of $A$ into $X$ extends to a homeomorphism of $X$~\cite[Theorem 3.2]{BesMog}.
\vspace{2pt}
\item If $X$ is $\mathcal C$-absorbing, then $X$ is $\mathcal F_0(X)$-absorbing,
where $\mathcal F_0(X)$ is the class of spaces homeomorphic to closed subsets 
of $X$~\cite[Proposition 1.11]{BanCau-interplay}. In particular, every closed
subset of $X$ admits a $Z$-embedding into $X$. 
\vspace{2pt}
\item If $\Om$ is an $\mathcal C$-absorbing AR, then one has:
\vspace{2pt}
\begin{itemize}
\vspace{2pt}
\item[\textup{(1)}]
At each point of a $\mathcal C$-absorbing space $X$ there is a local basis
of opens sets homeomorphic to $\Om$. 
\vspace{2pt}
\item[\textup{(2)}]
A space $X$ is $\mathcal C$-absorbing if and only if $X$ is an $\Om$-manifold.
\vspace{2pt}
\item[\textup{(3)}]
Any $\Om$-manifold is homeomorphic to an open subset of any other {$\Om$-manifold}.
\end{itemize}
\end{itemize}
The assertions (2)--(3) are from~\cite[Exercise 4 in 1.6]{BRZ-book}, and for 
completeness we now prove (1). 
If $X$ is $\mathcal C$-absorbing, then $X$ can be identified with a homotopy dense subset
of an $\ell^2$-manifold $L$,
so for any open $U\subset L$ the set $U\cap X$ is homotopy dense in 
$U$~\cite[Corollary 1.2.3]{BRZ-book}. Hence if $U$ is contractible, then so is $U\cap X$
(because the inclusion of a homotopy dense subset is a homotopy equivalence). 
Since $L$ has a local basis of contractible open sets, so does $X$.
Each of the contractible open sets is homeomorphic to $\Om$ by uniqueness of
homotopy equivalent $\mathcal C$-absorbing spaces.

\section{$\mathcal M_2$-universality of some functions spaces}
\label{sec: aux function space}

The main result of this section is Corollary~\ref{cor: Cbullet is M2-universal}
establishing $\M_2$-universality of a certain class of function spaces. 
Throughout this section $n$ is a nonnegative integer. 
Fix a manifold $V$ exhausted by an increasing sequence
of compact domains $D_j$.
As in Section~\ref{sec: Cgamma} we equip the linear space $C^n(V)$ of $n$-times continuously differentiable real valued
functions on $V$ with the Fr\'echet space structure given by
the family of seminorms 
\begin{equation}
\label{form: C^n seminorms}
\|f\|_{C^n(D_j)}=\max\{|f^{(k)}(x)|\co k\le n,\, x\in D_j\}.
\end{equation}
This family of seminorms also endows $C^\infty(V)=\bigcap_{n} C^n(V)$ with a Fr{\'e}chet space
structure.

Let $\M_0$ be the class of metrizable compacta, let $\M_1$ be the class of Polish (i.e., completely metrizable)
spaces, and let $\M_2$ be the class of  
topological spaces that are homeomorphic to $F_{\sigma\delta}$-sets in metrizable compacta.
Note that $\M_0\subset \M_1\subset \M_2$.
By $(\M_0,\M_2)$ we denote the class of pairs $(K,M)$ where $K$ is a compact metrizable space and $M$ is an $F_{\sigma\delta}$-set in $K$.

A space $X$ is called {\em $\M_2$-universal} if each space in $\M_2$ is homeomorphic to a closed 
subset of $X$. A pair $(Y,X)$ of spaces with $X\subset Y$ is called
\begin{itemize}
\item {\em $(\M_0,\M_2)$-universal\,} if for any pair $(K,M)\in (\M_0,\M_2)$ there exists a topological embedding $f:K\to Y$ such that $f^{-1}(X)=M$;\vspace{2pt} 
\item {\em $(\M_0,\M_2)$-preuniversal\,} if for any pair $(K,M)\in (\M_0,\M_2)$ there exists a continuous map $f:K\to Y$ such that $f^{-1}(X)=M$.
\end{itemize}
The following fundamental result is proved in~\cite[Theorems 3.1.1 and 3.2.11]{BRZ-book}:
{\it A subspace $X$ of a Polish space $Y$ is $\M_2$-universal if and only if 
 $(Y,X)$ is $(\M_0,\M_2)$-universal if and only if $(Y,X)$ is $(\M_0,\M_2)$-preuniversal.}

The following standard fact can be found in \cite[Exercises 12-13 in 1.2]{BRZ-book}.

\begin{lem}\label{lem:conv-hdense} Let $L$ be a locally convex space, $X$ be a metrizable convex set in $L$ and $D$ is a dense convex subset of $X$. Then $D$ is homotopy dense in $X$.
\end{lem} 

The following lemma is an immediate consequence of~\cite[Lemma 8.10]{CDM} and the previous lemma.

\begin{lem}
\label{lem: relative map CDM} 
Let $X$ be a Polish convex subset of a Fr\'echet space and $D$ be a dense convex subset of $X$ contained in a 
$\sigma Z$-set $Y\subset X$. Then for any compact metrizable space $K$ and a $\sigma$-compact subset $A\subset K$ there exists a continuous map $f:K\to X$ such that $f(A)\subset D$ and $f(K\setminus A)\subset X\setminus Y$.
\end{lem}

A subset $C$ of a linear topological space $L$ 
is called {\em $\infty$-convex\,} if for any bounded sequence $(x_n)_{n\ge 0}$ in $C$ and 
any sequence of nonnegative reals $t_n$ with $\sum_{n\ge 0} t_n=1$ the series 
$\sum_{n\ge 0} t_n x_n$ converges to some point of $C$. Recall that a subset $B$ of $L$ is 
{\em bounded\,} if for every neighborhood $U\subset L$ of zero 
$B$ is contained in all but finitely many sets $n\hspace{.5pt}U$, $n\in\N$.

Recall that a subspace $A$ of a Polish space $B$ is Polish if and only if $A$ is $G_\de$ in $B$.

\begin{thm}
\label{thm: bullet M2 univer} 
For an integer $l\ge 0$ let $C^l_\bu$ be an $\infty$-convex $G_\de$ subset of $C^l(V)$.
For $\g\in [l,\infty]$ set $C^\g_\bu=C^l_\bu\cap C^\g(V)$. If $n\ge l$ is an integer
such that for every integer
$k\ge n$ the subset $C^\infty_\bu$ is dense in $C^k_\bu$,
and $C^{k+1}_\bu$ is contained in a $\sigma Z$-subset of $C^k_\bu$,
then $(C^n_\bu,C^\infty_\bu)$ is $(\M_0,\M_2)$-preuniversal.
\end{thm}
\begin{proof} 
We shall apply the technique used in the proof of~\cite[Theorem 5.1, p.73]{Ban-sigma-to-the-omega}. 
Fix $n\ge l$ and a pair $(K,M)\in(\M_0,\M_2)$. 
We wish to construct a continuous map $f:K\to C^n_\bu$ 
such that $f(M)\subset C^\infty_\bu$ and $f(K\setminus M)\subset C^n_\bu\setminus C^{\infty}_\bu$.

Since $C^l_\bu$ is $G_\de$ in $C^l(V)$, for each integer $k\ge l$
the subset $C^k_\bu$ is $G_\de$ in $C^k(V)$, and in particular $C_\bu^k$ is Polish.
Since $M\in \M_2$ we can write $M=\bigcap_{k=n}^\infty A_k$ where $(A_k)_{k=n}^\infty$ is 
a decreasing sequence  of $\sigma$-compact subsets of $K$. 
By Lemma~\ref{lem: relative map CDM} 
for each $k\ge n$ there exists a continuous map $f_k:K\to C^k_\bu$ such that 
$f_k(A_k)\subset C^\infty_\bu$ and $f_k(K\setminus A_k)\cap C^{k+1}(V)=\emptyset$. 

Fix a basepoint $b\in C^\infty_\bu$; we also think of $b$
as a constant map taking $K$ to $\{b\}$.
Replacing $f_k$ by $(1-\e_k)b+\e_k f_k$ for a sufficiently small positive $\e_k$
we can additionally assume that 
\begin{equation}
\label{form: 2^-k}
f_k(K)\subset \{g\in C^k_\bu:\|g-b\|_{C^k(D_k)}\le 2^{-k}\}.
\end{equation}
Let us show that the series 
\begin{equation}
\label{form: series}
\sum_{k\ge n} 2^{-k} (f_k(z)-b)
\end{equation}
converges uniformly in $C^n(V)$ where $z\in K$. We give $C^n(V)$ 
the metric $\sum_{j\ge n} \frac{p_j}{p_j+1}2^{-j}$, where $p_j=\|\cdot\|_{_{C^n(D_j)}}$, and
check that the sequence of partial sums of the series (\ref{form: series}) is Cauchy.
Fix $j$, suppose $i>m>\max(n, j)$, and estimate 
\[
\|\sum_{k=m}^i 2^{-k} (f_k(z)-b)\|_{_{C^n(D_j)}}\le
\sum_{k=m}^i 2^{-k} \|f_k(z)-b\|_{_{C^k(D_k)}}\le 2^{-m}\sum_{k=m}^i 2^{-k}<2^{-m}
\]
where in the first inequality we used the definition (\ref{form: C^n seminorms}) and the fact
that $D_j\subset D_k$ for all $j<k$, while the second inequality depends on (\ref{form: 2^-k}).
This proves uniform convergence of (\ref{form: series}) and therefore defines a continuous
map $f:K\to C^n(V)$ given by 
\begin{equation}
\label{form: series inf convex}
f=b+2^{n-1}\cdot\sum_{k\ge n} 2^{-k} (f_k-b)=2^{n-1}\cdot\sum_{k\ge n} 2^{-k} f_k.
\end{equation}
Next we show that for every $z\in K$ the sequence $(f_k(z))_{k\ge l}$ is bounded in $C^n(V)$. 
Note first that 
\[
\|f_k(z)-b\|_{_{C^l(D_l)}}\le \|f_k(z)-b\|_{_{C^k(D_k)}}\le 2^{-k}\le 2^{-l} 
\]
so that
\[
\|f_k(z)\|_{_{C^l(D_l)}}\le 2^{-l}+\|b\|_{_{C^l(D_l)}}
\]
and since any neighborhood of $0$ in $C^n(V)$ contains a neighborhood
of the form $\{g\in C^n(V)\co \|g\|_{_{C^l(D_l)}}<\e\}$, the sequence $(f_k(z))_{k\ge l}$ 
is bounded in $C^n(V)$.

The $\infty$-convexity of $C^l_\bu$ in $C^l(V)$ implies 
that $C^\g_\bu$ is $\infty$-convex in $C^\g(V)$ for every $\g\in [l,\infty]$.
Since $f_k(z)\in C^n_\bu$ for each $z\in K$ and $k\ge n$, the right hand side of 
(\ref{form: series inf convex}) lies in $C^n_\bu$, i.e., $f(K)\subset C^n_\bu$.
By the same token $f(M)\subset C^\infty_\bu$.

It remains to prove that $f(z)\notin C^\infty(V)$ for every $z\in K\setminus M$. 
Fix such $z$ and find a unique $m>n$ with $z\in A_{m-1}\setminus A_{m}$. 
Thus $z\in A_k$ for $k<m$ and $z\notin A_k$ for $k\ge m$. 
The choice of $z$ and $f_k$ guarantee that 
\begin{itemize}
\item $f_m(z)\in C^{m}(V)\setminus C^{m+1}(V)$, \vspace{2pt}
\item $f_k(z)\in C^\infty(V)$ for all $k<m$, \vspace{2pt}
\item $f_k(z)\in C^{m+1}(V)$ for all $k>m$. 
\end{itemize}
In the right hand side of
\[
2^{1-n}\, f(z)=2^{-m} f_m(z)+ \sum_{k=n}^{m-1} 2^{-k}f_k(z)+ \sum_{k>m}2^{-k}f_k(z)
\]
the second summand is in $C^\infty(V)$ and the third summand converges 
in the Fr\'echet space $C^{m+1}(V)$. Since  $f_m(z)$ is not in $C^{m+1}(V)$
neither is $f(z)$.
\end{proof}

\begin{lem}
\label{lem: preuniv cont inj}
If $(A,B)$ is a $(\M_0,\M_2)$-preuniversal pair, and
$f\co A\to X$ is a continuous injective map to a Hausdorff topological space $X$,
then the subspace $f(B)$ of $X$ is $\M_2$-universal
and the pair $(X,f(B))$ is $(\M_0,\M_2)$-universal.
\end{lem}
\begin{proof} 
The pair $(A,B)$ is $(\M_0,\M_2)$-universal by the above-mentioned equivalence of 
universality and preuniversality. Thus for each pair
$(K,M)\in(\M_0,\M_2)$ there is a topological embedding $g:K\to A$ such that $g^{-1}(B)=M$. 
Since $K$ is compact and $X$ is Hausdorff, the continuous injective map $h=f\circ g\co K\to X$ is a topological embedding. The injectivity of $f$ guarantees that $h^{-1}(f(B))=g^{-1}(B)=M$.  
Thus the map $h$ witnesses $(\M_0,\M_2)$-universality of the pair $(X,f(B))$. We can choose the pair
$(K, M)$ so that $M$ is $\M_2$-universal (e.g., using that each separable metric space embeds into
the Hilbert cube). The homeomorphism $h\co K\to h(K)$ restricts to a homeomorphism
$M\to h(M)=f(B)$, so the latter is $\M_2$-universal.
\end{proof}

Theorem~\ref{thm: bullet M2 univer} and Lemma~\ref{lem: preuniv cont inj}
immediately imply the following corollary, which is
the main result of this section.

\begin{cor}
\label{cor: Cbullet is M2-universal}
Let $(C^n_\bu,C^\infty_\bu)$ be as in Theorem~\ref{thm: bullet M2 univer}.
If $f\co C^n_\bu\to X$ is a continuous injective map to a Hausdorff topological space $X$,
then the subspace $f(C^\infty_\bu)$ of $X$ is $\M_2$-universal.
\end{cor}

\section{$\M_2$-universal spaces of smooth functions}
\label{sec: M2-univ operator D}

In this section we give a construction of sequences $(C^n_\bu,C^\infty_\bu)_{n\ge l}$
satisfying the assumptions of Theorem~\ref{thm: bullet M2 univer} which covers all examples 
arising in the present paper, and also give a criterion of when a space of smooth maps
belongs to the class $\M_2$.

Let $N$ be a smooth manifold, possibly with boundary. With the $C^m$
topology the space $C^m(N)$ of $m$-differentiable functions on $N$
is a separable Fr\'echet space, 
where $m$ is a nonnegative integer or $\infty$, see Section~\ref{sec: Cgamma}. 

\begin{thm}
\label{thm: criterion M_2-univer}
Let $N$ be a smooth manifold, possibly with boundary, and 
let $D\subset \mathrm{Int}\, N$ be a smoothly embedded top-dimensional closed disk
that is mapped via a coordinate chart to a Euclidean unit disk. Let $l\ge 0$ be an integer and 
let $\mathfrak D\co C^l(N)\to C^0(N)$ be a continuous linear map.
Given $\eta\in\R$ suppose there exists $h_\bu\in C^\infty(N)$ with 
$\mathfrak D h_\bu\vert_{_D}>\eta$.
Let $C^l_\bu$ denote the subspace of $C^l(N)$ of functions $u$ 
such that $\mathfrak D u\vert_D\ge \eta$ and $u\vert_{N\setminus \mathrm{Int}\,D}=h_\bu$. 
Then for $C^n_\bu=C^l_\bu\cap C^n(N)$ and $C^\infty_\bu=C^l_\bu\cap C^\infty(N)$
the sequence $(C^n_\bu,C^\infty_\bu)_{n\ge l}$ satisfies the
assumptions of Theorem~\textup{\ref{thm: bullet M2 univer}}.
\end{thm} 
\begin{proof}
For every $n\ge l$ the subset $C^n_\bu$ is clearly convex and closed
in the Fr\'echet space $C^n(N)$, so in particular, $C^n_\bu$ is $G_\de$ in $C^n(N)$.

The $\infty$-convexity of $C^l_\bu$ follows from the fact that 
any closed convex subset of a Fr{\'e}chet space is $\infty$-convex.
(To prove the fact suppose $x_n$, $t_n$, $C$, $L$ are as in 
the definition of $\infty$-convexity given before Theorem~\ref{thm: bullet M2 univer}, and 
the Fr{\'e}chet structure on $L$ 
is given by a countable family of seminorms $p_k$. Since $\{x_n\}$ is bounded, 
$p_k(x_n)$ is bounded by a constant depending only on $k$, which
easily implies that the partial sums of the series $\sum_{n\ge 0} t_n x_n$ form a Cauchy
sequence in $L$. Their limit in $L$ is also the limit of
the sequence $(\sum _{n=0}^m t_n)^{-1}\sum_{n=0}^m t_n x_n$ which lies in $C$).

For a positive integer $m$ and a point $p\in D$
let $Z_m$ be the set of all $u\in C^{n+1}_\bu$
whose $C^k(p)$ norm is $\le m$ (the norm is simply the sum of the 
absolute values of all partial derivatives of $u$
of orders $\le k$ evaluated at $p$). It is easy to check that 
$Z_m$ is a $Z$-set, see the proof of Lemma~\ref{lem: O is sigmaZ},
and clearly $C^{n+1}_\bu$ equals in the union of the sets $Z_m$. 

To show that $C^\infty_\bu$ is dense in $C^n_\bu$ 
fix $u\in C^n_\bu$.
Replacing $u$ with $\delta h_\bu+(1-\delta) u$, 
where $\delta$ is small and positive,
we can assume that  $\mathfrak D u\vert_{_D}>\eta$. 
Henceforth we suppress distinctions between
the objects in the Euclidean space and their images under the coordinate chart
containing $D$. In the chart we fix a concentric closed disk $D_\e\subset \Int\,D$
such that $D$, $D_\e$ have radii $1$, $1-\e$, respectively. 
Consider the partition of unity $\{\psi_\e, 1-\psi_\e\}$ 
such that $\psi_\e\vert_{D_\e}=1$, the support of $\psi_\e$ equals $D$, and 
$\psi_\e$ is rotationally symmetric. We also assume that the $k$-th radial derivative
of $\psi_\e$ grows with $\e$ as $\e^{-k}$ which can be arranged as follows:
Start with a smooth function $\psi\co\R\to [0,1]$ with
$\psi\vert_{[-\infty,1]}=1$ and $\psi\vert_{[2,\infty]}=0$, 
and set $\psi_\e(r)=\psi(2-\frac{1-r}{\e})$ for $r\ge 0$.

Approximate $u\vert_{D}$ by a family $\{u_\tau\}_{\tau\ge 0}$
of smooth maps such that $u_\tau\to u\vert_{D}$ in $C^n(D)$ as $\tau\to 0$.
Then glue $u_\tau$ and $h_\bu$ via the partition of unity $\{\psi_\e, 1-\psi_\e\}$. 
The result $(1-\psi_\e)h_\bu+\psi_\e u_\tau\in C^\infty(N)$ equals $u_\tau$ on $D_\e$ and $h_\bu$ on 
$M\setminus\Int\, D$. 

We are going to show that $(1-\psi_\e)h_\bu+\psi_\e u_\tau$
approximates $u$ in $C^n(N)$ for suitable $\e$, $\tau$. 
Since $\mathfrak D u\vert_{_D}>\eta$ this would imply  
$(1-\psi_\e)h_\bu+\psi_\e u_\tau\in C_\bu^\infty$.
Write
\[
(1-\psi_\e)h_\bu+\psi_\e u_\tau-u=(1-\psi_\e)(h_\bu-u)+\psi_\e(u_\tau-u).
\]
On the right hand side for every fixed $\e$ the second summand
converges to zero in $C^n(D)$ as $\tau\to 0$, hence it suffices to show
that the first summand converges to zero in $C^n(D)$ as $\e\to 0$.

Set $v=h_\bu-u$. Note that $v\in C^n(N)$ 
and $v$ vanishes on $M\setminus\Int\, D$; in particular, all derivatives
of $v$ of orders $\le n$ vanish on $\d D$.
Arguing by contradiction suppose that there is a sequence
of points of $D$ where the derivatives
of $v(1-\psi_\e)$ of some order $k\le n$ do not approach zero as $\e\to 0$. By compactness
we may assume that the points converge to a point $p$, which has to lie on $\d D$.
Introduce new coordinates $(r,\th)$ near $p$ where $r$ is the distance to $\d D$
and $\th$ represents the remaining coordinates along $\d D$.
Since $\psi_\e$ does not depend on $\th$, the only derivatives of $v(1-\psi_\e)$
that might not approach zero are the derivatives by $r$. Hence we
need to analyze the terms $(1-\psi_\e)^{(k-m)}\frac{\d^m v}{\d r^m} $
for $0\le m< k$; the case $m=k$ is trivial since $1-\psi_\e$ is bounded independently of $\e$. 
Consider the order $k-m$ Taylor expansion of $\frac{\d^m v}{\d r^m}$ 
at $p$ with respect to $r$. The only (possibly) nonzero term is 
the remainder $\frac{r^{k-m}}{(k-m)!}\frac{\d^k v}{\d r^k}(\xi, \th)$ where $0<\xi<r$.
We can restrict attention to $r<\e$ because $1-\psi_\e$ vanishes for $r\ge \e$.
By the choice of $\psi_\e$ the term $|(1-\psi_\e)^{(k-m)}|\e^{k-m}$ is bounded,
so there is a constant $C>0$ such that 
\[
\left|(1-\psi_\e)^{(k-m)}\frac{\d^m v}{\d r^m} \right|\le 
\left|(1-\psi_\e)^{(k-m)}\frac{\d^k v}{\d r^k}(\xi, \th)\frac{\e^{k-m}}{(k-m)!}\right|\le
C\left|\frac{\d^k v}{\d r^k}(\xi, \th)\right|
\]
where the right hand side approaches zero as $\e\to 0$ because $0<\xi<\e$. 
This contradiction completes the proof that
$C^\infty_\bu$ is dense in $C^n_\bu$. 
\end{proof}

Let $\G^\g(V, \R^n)$ denotes the space of smooth maps from a manifold $V$ to $\R^n$
endowed with the $C^\g$ topology, $\g\in [0,\infty]$.
Recall that $\M_2$ is the class of spaces 
homeomorphic to $F_{\sigma\delta}$-sets in metrizable compacta.

\begin{lem}
\label{lem: O is in M2} For every $\g$
the identity map $\id\co\G^\infty(V, \R^n)\to\G^\g(V, \R^n)$
takes any $G_\de$ subset to a space in $\M_2$. 
\end{lem}
\begin{proof}
Fix a $G_\de$ subset $S^\infty$ of $\G^\infty(V, \R^n)$
and set $S^\g=\id(S^\infty)$.  
Recall that a (not necessarily continuous) 
map is {\em $F_\s$-measurable\,} if the preimage of any open set is $F_\s$.
To prove the lemma it suffices to show that the inverse of the
identity map $\id\co S^\infty\to S^\g$
is $F_\s$-measurable because in this case
the fact that $\id$ is continuous (and hence $F_\s$-measurable) implies that 
the map $\id$ is a $(1,1)$ homeomorphism in the terminology 
of~\cite[Corollary 3 in VII, section 35]{Kur}, where it is shown that
any $(1,1)$ homeomorphism maps a completely metrizable separable space (such as $S^\infty$)
to a space in $\M_2$.

Since the restriction and the enlarging of a co-domain of an $F_\s$-measurable map is $F_\s$-measurable,
it is enough to show that the inverse of the identity map $\id\co\G^\infty(V, \R^n)\to \G^\g(V, \R^n)$
is $F_\s$-measurable.
 
If $Y$ is separable and metrizable, then to prove that a map $f\co X\to Y$
is $F_\s$-measurable 
it (clearly) suffices to show that for each open set $W\subset Y$ and each $y\in W$ there is 
a subset $T_y\subset W$ with $y\in\Int\, T_y$ such that $f^{-1}(T_y)$ is closed.
Let us apply this to $f=\id\co\G^\g(V, \R^n)\to \G^\infty(V, \R^n)$.  

Fix an exhaustion of $V$ by compact subsets $D_j$, fix some norm on $\R^n$,
and denote by $\|\cdot\|_{_{C^m(D_j)}}$ the associated $C^m$ norm. 
Given any $m>\g$ note that
\[
T_y=y+\{f\in \G^\g(V, \R^n)\co ||f||_{_{C^m(D_j)}}\le \e\}
\] 
is closed in $\G^\g(V, \R^n)$ and $y\in\Int\, T_y$. Also 
$T_y$ lies inside any neighborhood of $y$ in $\G^\infty(V, \R^n)$ if $\e$
is small enough.
\end{proof}

\section{Space of all complete Riemannian metrics: smooth topology}
\label{sec: R smooth}

In this section we prove Theorem~\ref{thm: all metrics} for $\g=\infty$.
Let $V$ be a smooth connected manifold of positive dimension.
The space $\mathcal R^{\g}(V)$ of all complete smooth Riemannian metrics on a manifold $V$ is
a convex subset in the metric linear space $\mathcal T^\g(V)$ of smooth symmetric $2$-tensors
on $V$ with the $C^\g$ topology. 
%Furthermore, $\mathcal R^{\g}(V)$ is open if $V$ is compact,
%and $\mathcal R^{\g}(V)$ has empty interior if $V$ is noncompact, see~\cite{FegMil}.

\begin{lem}
\label{lem: R is G_delta}
$\mathcal R^{\g}(V)$ is a $G_\delta$ subset of $\mathcal T^\g(V)$. 
\end{lem}
\begin{proof}
Let $(K_n)$ be a countable family of compact subsets exhausting $V$. Then
$\mathcal R^{\g}(V)$ is the intersection of the
the sets $\{t\in \mathcal T^\g(V)\co t\vert_{K_n}\ \text{is positive definite}\}$
which are open in  $\mathcal T^\g(V)$.
\end{proof}

\begin{rmk}
$\mathcal R^{\g}(V)$ is (clearly) open if $V$ is compact, and 
has empty interior if $V$ is noncompact, see~\cite{FegMil}.
\end{rmk}

Note that $\mathcal T^\infty(V)$ is Fr\'echet. To show that $\mathcal R^{\infty}(V)$
is homeomorphic to $\ell^2$ it is enough to check that the closure of
$\mathcal R^{\infty}(V)$ in $\mathcal T^\infty(V)$
is not locally compact~\cite[Theorem 2]{DobTor}. The closure is the convex
set $\mathcal N^{\infty}(V)$ of smooth symmetric non-negative definite $2$-tensors.

\begin{lem}
$\mathcal N^{\infty}(V)$ is not locally compact.
\end{lem}
\begin{proof} This is a variation of~\cite[Lemma 2.5]{BelHu-modr2}.
Suppose arguing by contradiction that every point of 
$\mathcal N^{\infty}(V)$ has a compact neighborhood and let $K$ be
a compact neighborhood of $g\in \mathcal R^{\infty}(V)$ in $\mathcal N^{\infty}(V)$.
Fix a disk $D\subset V$ inside a coordinate chart, and consider
the restriction map $\de\co \mathcal N^{\infty}(V)\to \mathcal T^{\infty}(D)$. 
Continuity of $\de$ implies that $\de(K)$ is compact. 

Let us show that $\de(K)$
is also a neighborhood of $\de(g)$. If not, then there is a sequence $\check\tau_i\notin\de (K)$ 
converging to $\de(g)$ in $\mathcal T^{\infty}(D)$. Extend $\check\tau_i$ to 
$\tau_i\in \mathcal T^\infty(V)$ converging to $g$.
Let $\phi$ a bump function on $V$
with $\phi\vert_D=1$. Then $g+\phi(\tau_i-g)$ converges to $g$ and hence
lies in $K$ for large $i$. The restriction of $g+\phi(\tau_i-g)$ to $D$ equals $\check\tau_i$
so $\check\tau_i\in \de(K)$, which is a contradiction.

Thus $\de(g)$ has a compact neighborhood in $\mathcal T^{\infty}(D)$, which
is a Fr\'echet space. Hence any point in the Fr\'echet space has a compact neighborhood,
which implies that $\mathcal T^{\infty}(D)$ is finite 
dimensional~\cite[Theorem 9.2]{Tre-TVS-book}, which is clearly false.
This contradiction proves that $\mathcal N^{\infty}(V)$ is not locally compact. 
\end{proof}

\section{Space of all complete Riemannian metrics: $C^\g$ topology}
\label{sec: R C^gamma}

In this section we prove Theorem~\ref{thm: all metrics} for finite $\g$.

\begin{lem} 
\label{lem: R is sigmaZ} 
If $\g$ is finite, then $\mathcal R^\g(V)$ is $\s Z$.
\end{lem}
\begin{proof}
With minor modifications the argument of Lemma~\ref{lem: O is sigmaZ}
works for $\mathcal R^\g(V)$, namely we 
choose $\a$ to take values in the space of nonnegative definite $2$-tensors on $V$
which ensures that $f(q)+\a\in\mathcal R^\g(V)$ for all $q$.
\end{proof}

\begin{lem}
\label{lem: R is in M2}
$\mathcal R^\g(V)$ belongs to $\M_2$.
\end{lem}
\begin{proof}
By Lemma~\ref{lem: R is G_delta} the subspace $\mathcal R^\infty(V)$ is $G_\de$
in the linear space of all smooth $2$-tensors on $V$ with $C^\infty$ topology,
so $\mathcal R^\g(V)$ is completely metrizable so that 
Lemma~\ref{lem: O is in M2} applies.
\end{proof}

\begin{lem}
\label{lem: R is M2-univ}
$\mathcal R^{\g}(V)$ is $\M_2$-universal for every finite $\g$.
\end{lem}
\begin{proof}
It is enough to show that $\mathcal R^{\g}(V)$ contains a closed $\M_2$-universal
subspace.
Fix any integer $l\ge \g$ and 
let $\mathfrak D\co C^l(V)\to C^0(V)$ be the inclusion. Let $h_\bu\in C^\infty(V)$ be the constant 
function that equals $2$ at all points, let $\eta=1$, and let $D$ be an embedded coordinate disk in $V$. 
Let $C^l_\bu$ be the set of functions $u\in C^l(V)$ with 
such that $u\vert_D\ge \eta$ and $u\vert_{V\setminus \mathrm{Int}\,D}=h_\bu$.
Let $X$ be the space of all metrically complete $C^l$-Riemannian metrics on $V$ with the $C^\g$ topology.
By Theorem~\ref{thm: criterion M_2-univer} and Corollary~\ref{cor: Cbullet is M2-universal}
for any continuous injective map from $f\co C^l_\bu\to X$ the space
$f(C^\infty(V)\cap C^l_\bu)$ is $\M_2$-universal. 

To define $f$ fix $g\in \mathcal R^{\infty}(V)$
and let $f(u)=u g$. The map $f$ is continuous as $l\ge\g$ and injective because $g$ is positive definite. 
Each metric in $f(C^\infty(V)\cap C^l_\bu)$ is smooth because $u$ is smooth,
and complete since $ug=2g$ outside $D$ and $2g$ is complete.

It remains to show that $f(C^\infty(V)\cap C^l_\bu)$ is closed in $\mathcal R^{\g}(V)$.
To this end suppose $u_ig\in f(C^\infty(V)\cap C^l_\bu)$ converges in $C^\g$ topology to 
a smooth complete metric $g_\ast$. Thus $g_\ast=2g$ outside $D$.
Choose a $g$-unit vector field $U$ on a neighborhood $W$ of $D$.
In the $C^\g$ topology on $W$
the functions $u_i=u_i g(U,U)$ converge to $g_\ast(U,U)$, which equals $2$ on $W\setminus D$.
Let $u$ be a function that equals $g_\ast(U,U)$ on $W$ and equals $2$ outside $D$; clearly $u$ is smooth.
Hence the metrics $u_ig$ converge to $ug$ in the $C^\g$ topology on $V$. By uniqueness of the limit
$g_\ast=ug$.
\end{proof}

\begin{rmk} The same proofs show that the space of all (complete or not) smooth Riemannian metrics
on a manifold $V$
equipped with the $C^\g$ topology is homeomorphic to $\ell^2$ when $\g=\infty$ and $\M_2$-universal if  $\g$ is finite.\end{rmk}

\begin{thm}
\label{thm: R is absorbing}
$\mathcal R^{\g}(V)$ is $\M_2$-absorbing, and hence is homeomorphic to $\Si^\o$.  
\end{thm}
\begin{proof}
Let $\overline{\mathcal T}^\g(V)$ be the closure of 
$\mathcal T^\g(V)$ of Section~\ref{sec: R smooth} in the Fr{\'e}chet 
space of all $C^\g$ $2$-tensors on $V$ with the $C^\g$ topology;
thus $\overline{\mathcal T}^\g(V)$ is a separable Fr{\'e}chet space. 
The closure of $\mathcal R^\g(V)$ in $\overline{\mathcal T}^\g(V)$ 
is not contained in $\mathrm{Aff}\, \mathcal R^\g(V)$
because it clearly contains some non-smooth metrics. 
Hence by Lemma~\ref{lem: convex embeds as homot dense} $\mathcal R^\g(V)$ is homeomorphic
to a homotopy dense subset of $\ell^2$. 
The proof of Lemma~\ref{lem: R is M2-univ} actually shows that 
the $\M_2$-universal subset $f(C^\infty(V)\cap C^l_\bu)$
is closed in $\mathcal T^\g(V)$. Therefore~\cite[Proposition 5.3.5]{BRZ-book}
implies that $\mathcal R^{\g}(V)$ is strongly $\M_2$-universal.
Thus Lemmas~\ref{lem: R is sigmaZ} and~\ref{lem: R is in M2}
imply that $\mathcal R^{\g}(V)$ is $\M_2$-absorbing. Since $\mathcal R^{\g}(V)$
is contractible, it is homeomorphic to $\Si^\o$.
\end{proof}

\section{Diffeomorphism groups and uniformization of $S^2$ and $RP^2$}
\label{sec: unif M}

Let $g_{_1}$ denote the standard curvature $1$ metric on $S^2$ and $RP^2$,
the unit sphere and its quotient by the antipodal map.
Henceforth we identify $S^2$ with $\overline{\mathbb C}$ so that
the antipodal map corresponds to $z\to-\bar z^{-1}$, and
the group of orientation-preserving conformal transformation of $(S^2, g_{_1})$ 
corresponds to $PSL(2,\mathbb C)$, which acts simply transitively on triples of distinct points of $S^2$.
Fix any triple of distinct points of $S^2$, and let $D^{\b}(S^2)$ be
the group of smooth orientation-preserving diffeomorphisms of $S^2$ fixing the triple and equipped with
the $C^{\b}$ topology, $\b\in [1,\infty]$. 

A diffeomorphism of $RP^2$ has two lifts to the orientation cover: a lift homotopic to the identity of $S^2$
and its composite with the antipodal map. So any 
diffeomorphism of $RP^2$ lifts to a unique orientation-preserving diffeomorphism of $S^2$
which by uniqueness must commute with the antipodal map. Since lifts of conformal diffeomorphism are conformal,
the orientation-preserving lifts of conformal diffeomorphisms of $(RP^2, g_{_1})$ are
precisely the elements of $PSL(2,\mathbb C)$ commuting with the antipodal map.
It is straightforward to check that
these lifts form the subgroup $PSU(2)\subset PSL(2,\mathbb C)$, which corresponds under the
above identification to $SO(3)$ acting on the unit sphere, and which 
acts simply transitively on the unit tangent bundle to $RP^2$.
%indeed, if $\pi\co S^2\to P^2$ is the projection, $d$ is a diffeomorphism of $P^2$, and $\tilde d$
%is a lift of $d$, then $\pi(\tilde d(v))=d(\pi(v))=d(\pi(-v))=\pi(\tilde d(-v))$
%so $\tilde d(v)=\pm \tilde d(-v)$ and since $\tilde d$ is injective $+$ does not occur.
Fix a vector in the unit tangent bundle of $RP^2$ and 
let $D^{\b}(RP^2)$ be the group of all diffeomorphisms of $RP^2$
that fix the vector.

If $M$ is $S^2$ or $RP^2$, then by
the Uniformization Theorem any metric on $M$ is of the form $\phi^*e^{-2u}g_{_1}$
where $u\in C^\infty(M)$ and $\phi$ is a diffeomorphism of $M$. 
(The standard accounts such as e.g.~\cite[Chapter 10]{Don} omit the case of
$RP^2$ which we include for completeness. Let $(S^2, \phi^*e^{-2u}g_{_1})$ be the pullback of a given metric 
$(RP^2, g)$ to the orientation cover. The covering isometry of $\phi^*e^{-2u}g_{_1}$
gives rise to an isometric involution $\iota$ of $e^{-2u}g_{_1}$.
Hence $\iota^*g_{_1}=e^{2(u\circ\iota-u)}g_{_1}$, and since the left hand side
has a transitive isometry group $u\circ\iota-u$ is constant, which we denote $\l$.
Precomposing $u\circ\iota-u=\l$ with $\iota$ gives $u-u\circ \iota=\l$, so $\l=0$.
Hence $\iota$ is an isometry of $g_{_1}$ and $u$ is $\iota$-invariant.
It follows that $e^{-2u}g_{_1}$ descends to a metric on $RP^2$ of desired form).

It is routine to check that $D^{\b}(M)$ is a topological group. 
(For $\b\notin\mathbb Z$ the relevant properties of H\"older map and inverses
of H\"older diffeomorphisms can be found, e.g., in~\cite[Section 2.2]{BHS-holder}).

Thus if $M$ equals $S^2$ or $RP^2$ then
no nontrivial element of $D^{\b}(M)$ is conformal, and any
orientation-preserving diffeomorphism of $M$ can be written uniquely as the composite
of an element of $D^{\b}(M)$ followed by a conformal automorphism.
%(Lickorish, Homeomorphisms of nonorientable 2-manifolds, 1963

For $\g\in [0,\infty]$ let
$\G^\g(M)$ denote the linear space of all smooth functions on $M$ with the $C^\g$
topology; thus as a set $\G^\g(M)$ equals $C^\infty(M)$.  

For $M$ equal to $S^2$ or $RP^2$ 
the above discussion shows that the map 
\begin{equation}
\label{form: bijection all}
D^{\g+1}(M)\times\G^\g(M)\to \mathcal R^{\g}(M)
\end{equation} 
given by $(\phi, u)\to \phi^*e^{-2u}g_{_1}$ 
is a continuous bijection.

The sectional curvature of $e^{-2u}g_{_1}$ equals $e^{2u}(1+\triangle_{g_{_1}} u)$
where $\Delta_{g_{_1}}$ is the $g_{_1}$-Laplacian, see~\cite[p.15, (1.3)]{KazWar-comp}.

Let $\O^{\g}_{\ge\l}(M)\subset \G^\g(M)$ be the subset of functions $u$ 
such that $e^{-2u}g_{_1}$ has sectional curvature $\ge \l$;
the subset $\O^{\g}_{>\l}(M)$ is defined similarly. 

\begin{lem}
\label{lem: convexity of exp}
The subsets $\O^{\g}_{\ge\l}(M)$ and $\O^{\g}_{>\l}(M)$ are convex in $\G^\g(M)$.
\end{lem}
\begin{proof}
Let $u=tu_{_1}+(1-t)u_{_0}$ where $t\in [0,1]$ and $u_{_0}, u_{_1}\in \O^{\g}_{\ge\l}(M)$.
Then 
{\small\[
1+\triangle_{g_{_1}} u=t(1+\triangle_{g_{_1}} u_{_1})+(1-t)(1+\triangle_{g_{_1}} u_{_0})\ge 
\l(te^{-2u_{_1}}+(1-t)e^{-2u_{_0}})\ge \l e^{-2u}
\]}\par
where the last inequality holds by convexity of the exponential function.
Thus the sectional curvature of $e^{-2u}g_{_1}$ is $e^{2u}(1+\triangle_{g_{_1}} u)\ge\l$.
\end{proof}
The map (\ref{form: bijection all}) restricts to the continuous bijections
{\small\begin{equation}
\label{form: bijection >=0}
D^{\g+1}(M)\times\O^\g_{\ge \l}(M)\to \mathcal R_{\ge \l}^{\g}(M)
\qquad\ \ 
D^{\g+1}(M)\times\O^\g_{>\l}(M)\to \mathcal R_{>\l}^{\g}(M).
\end{equation}}\par
\begin{thm}
\label{thm: unif}
If $\g\notin\mathbb Z$, then the maps 
\textup{(\ref{form: bijection all})} and \textup{(\ref{form: bijection >=0})}
are homeomorphisms, and in particular, $D^{\g+1}(M)$, $\mathcal R_{\ge \l}^{\g}(M)$,
$\mathcal R_{>\l}^{\g}(M)$ are contractible.
\end{thm}
\begin{proof}
This is a minor modification of~\cite[Theorem 4.1]{BelHu-modr2} whose statement incorrectly assumes
$\g\in\mathbb Z$ instead of $\g\notin\mathbb Z$, cf.~\cite{BelHu-modr2-err}. The proof is the same
and the main ingredient is that the solution of Beltrami equation in a chart depends
in H\"older topology on the dilatation, which is where $\g\notin\mathbb Z$ is used. 
Relevant properties of H\"older map can be found, e.g., in~\cite[Section 2.2]{BHS-holder}.
Since $D^{\g+1}(M)$ is a retract of 
the convex set $\mathcal R^{\g}(M)$, it is contractible. This together with convexity of $\O_{\ge\l}^\g(M)$ and $\O_{>\l}^\g(M)$
implies contractibility of $\mathcal R_{\ge\l}^{\g}(M)$ and $\mathcal R_{>\l}^{\g}(M)$.
\end{proof}

\begin{rmk}
S.~Smale proved in~\cite{Sma-2sph} that $D^\b(S^2)$ is contractible for every integer $\b>1$.
The same should be true for $\b=1$, as well as for $D^{\b}(RP^2)$ where $\b$ is a positive integer, 
but we have not checked the details
and are not aware of any proof in the literature. 
\end{rmk}

\begin{cor} 
\label{cor: diff loc homeo to normed space}
Let $M$ be $S^2$ or $RP^2$. Then $D^\infty(M)$ is homeomorphic to $\ell^2$.
If $\b\in [1,\infty)$, then
$D^{\b}(M)$ is locally homeomorphic to a normed space so that
$D^{\b}(M)$ is an ANR.
\end{cor}
\begin{proof}
Since $D^{\b}(M)$ is topologically homogeneous it is enough
to consider a neighborhood of the identity. Any diffeomorphism $\phi$ of $M$ 
that is sufficiently close to the identity can be written as 
$\phi(p)=\exp_p X_\phi(p)$ where $\exp$ is the $g_{_1}$-exponential map and
$X_\phi$ is a smooth vector field. The map $\phi\to X_\phi$ defines a topological embedding of
a neighborhood of the identity of $D^{\b}(M)$ into the linear space of smooth vector fields
on $M$ with the $C^{\b}$ topology. 
Conversely, every $C^\b$ vector field on $M$ can be exponentiated to a
$C^\b$ self-map of $M$. If $M=S^2$, then $\phi$ fixes
a triple of points if and only if $X_\phi$ vanishes at the points.
If $M=RP^2$, then $\phi$ fixes a point tangent vector $v$ at $p\in M$ if and only if $X_\phi(p)=0$
and the differential of $X_\phi$ fixes $v$ (where we used that the differential of $\exp_p$
is the identity). We denote by $\X^{\b}(M)$ the linear space
of smooth vector fields on $M$ equipped with
the $C^{\b}$ norm and subject to the above vanishing conditions.
Note that $\X^{\b}(M)$ is normed for finite $\b$ and Fr{\'e}chet for $\b=\infty$.
Since $\b\ge 1$ and $M$ is compact, $C^\b$ diffeomorphisms form an open subset
in $C^\b(M,M)$, so a sufficiently small neighborhood of zero in $\X^{\b}(M)$ exponentiates
to a neighborhood of the identity in $D^\b(M)$. 
If $\b=\infty$, then like any separable Fr{\'e}chet space $\X^{\b}(M)$ is homeomorphic to $\ell^2$,
and hence so is $D^\infty(M)$ being a contractible $\ell^2$-manifold.
If $\b$ is finite, the identity in $D^{\b}(M)$ has a neighborhood homeomorphic
to an open ball in $\X^{\b}(M)$, and 
since any open ball and the ambient normed space are homeomorphic, the claim follows.
Finally, any normed space is an AR, and any locally 
ANR space is an ANR, see~\cite[Section II.5]{BP-book}.
\end{proof}

\section{Uniformization of nonnegatively curved planes}
\label{sec: unif of C}

Let $g_{_0}$ be the standard Euclidean metric on $\R^2$, which we identify with $\C$.
Orientation-preserving conformal automorphisms of $\C$ are precisely the affine maps
$z\to az+b$, $a,b\in\C$; the group acts simply transitively on pairs of distinct points 
of $\C$. Fix such a pair and let $D^\b(\C)$ be
the group of smooth orientation-preserving diffeomorphisms of $\C$ fixing the pair and equipped with
the $C^{\b}$ topology. 

As in Section~\ref{sec: unif M} we conclude that
$D^\b(\C)$ is a topological group, but the proof of 
Corollary~\ref{cor: diff loc homeo to normed space}
fails because $D^\b(\C)$ is no longer open in $C^\b(\C, \C)$.
It is shown in~\cite{Yag} that $D^\infty(\C)$ is homeomorphic to $\ell^2$, and we adapt his
argument for $\b\in [1,\infty)$ in Lemma~\ref{lem: diff C is ANR} 
to show that $D^\b(\C)$ is an ANR. Contractibility of $D^\infty(\C)$
is proved in Lemma~\ref{lem: diff C is contractible} by exhibiting
a explicit deformation retraction to a point.

\begin{rmk}
The proof of contractibility of $D^\b(S^2)$, $\b\notin\Z$, given in Section~\ref{sec: unif M}
does not work for $D^\b(\C)$ because metrics on $\C$ are not all conformally equivalent, and
hence it is unclear how to define an analog of the map (\ref{form: bijection all}).
\end{rmk}

A basic property of a subharmonic function $u$ is that the limit
\[
\a(u):=\lim_{r\to\infty} \frac{\sup\{u(z)\,:\,|z|=r\}}{\log r}
\]
exists in $[0,\infty]$. It is proved in~\cite[Theorem 1.1]{BelHu-modr2} that the nonnegatively curved metric
$e^{-2u}g_{_0}$ is complete if and only if $\a(u)\le 1$.

Let $\O^\g_{\ge 0}(\C)$ be the set of smooth subharmonic functions on $\C$
with $\a(u)\le 1$ equipped with the $C^\g$ topology, and let  $\O^\g_{>0}(\C)$
be the subspace of functions $u$ such that $\triangle_{g_{_0}} u$ is positive.
Both $\O^\g_{\ge 0}(\C)$ and $\O^\g_{>0}(\C)$ are convex 
by the proof of~\cite[Lemma 2.4]{BelHu-modr2}.

It is well known, see e.g.~\cite{BlaFia} that any complete nonnegative curved plane
is conformally equivalent to $(\C, g_{_0})$. This easily implies 
that the map $(\phi, u)\to \phi^*e^{-2u}g_0$ induces continuous bijections
{\small\begin{equation}
D^{\g+1}(\C)\times \O^\g_{\ge 0}(\C)\to\mathcal R_{\ge 0}^{\g}(\C)\qquad\ \ \
D^{\g+1}(\C)\times \O^\g_{>0}(\C)\to\mathcal R_{>0}^{\g}(\C)
\end{equation}}\par
which are homeomorphisms if $\g\notin\mathbb Z$, see~\cite[Theorem 4.1]{BelHu-modr2} 
and~\cite{BelHu-modr2-err}.

Let $M$ be $S^2$, $RP^2$ or $\C$ equipped with a complete metric $g_\k$ of constant curvature $\k\in\{0,1\}$, 
and as in Section~\ref{sec: unif M} we let $\G^\g(M)$ denote the locally convex linear metric space of all smooth functions on $M$ with the $C^\g$ topology.

\begin{lem} 
\label{lem: C GH closed}
Let $M$ be $S^2$, $RP^2$ or $\C$, and if $M=\C$ assume $\l=0$. 
Then $\O^\b_{\ge \l}(M)$ is closed in $\G^\b(M)$.
\end{lem}
\begin{proof}
The convergence of $u_i\in\O^\b_{\ge \l}(M)$ to $u\in\G^\b(M)$ gives rise
to the pointed Gromov-Hausdorff convergence $e^{-2u_i}g_\l\to e^{-2u}g_\l$ of the corresponding
metrics. 
If $M$ is compact, then $(M, e^{-2u}g_{_1})$ has 
curvature $\ge\l$  in the comparison sense, and hence sectional curvature $\ge \l$,
i.e., $u\in \O^\b_{\ge \l}(M)$.
If $M=\C$, then a priori the limiting metric $e^{-2u}g_0$ might not be complete,
but in any case the completion of $(\C, e^{-2u}g_0)$
is an Alexandrov space of nonnegative curvature in 
the comparison sense. Like any Riemannian manifold, $(\C, e^{-2u}g_0)$ is a locally convex subset
of its completion, so it is also an Alexandrov space of nonnegative curvature,
and since $u$ is smooth the metric $e^{-2u}g_0$
has nonnegative sectional curvature, i.e., $\triangle_{g_{_0}} u\ge 0$. 
Since $u_i\in\O^\b_{\ge \l}(M)$ we have $\triangle_{g_{_0}} u_i\ge 0$ and $\a(u_i)\le 1$.
Now the proof of~\cite[Lemma 2.4]{BelHu-modr2} applies to show that $\a(u)\le 1$
so that $u\in \O^\g_{\ge 0}(\C)$.
\end{proof}

\section{Recognizing $\ell^2$-manifolds: smooth topology}
\label{sec: O smooth}

In this section we prove Theorems~\ref{thm: main compact}--\ref{thm: main C} for $\g=\infty$.
Let $M$ denote $\C$, $S^2$, or $RP^2$ equipped with a metric $g_\k$ of constant 
curvature $\k\in\{0,1\}$. 

That $\O_{\ge 0}^\infty(\C)$ is not locally compact is proved in~\cite[Lemma 2.5]{BelHu-modr2}
where the idea was to look at the point of $\C$ where the curvature of $e^{-2u}g_{_0}$ is positive
and note that arbitrary small perturbations near the point rule out local compactness. 
The same idea works for $\O_{\ge \l}^\infty(M)$ when $M$ is compact. Namely, if $u$
is a constant function on $M$ such that the sectional curvature of $e^{-2u}g_{_1}$
is $>\l$, then the proof of~\cite[Lemma 2.5]{BelHu-modr2} shows that $u$ has no compact neighborhood
in $\O_{\ge \l}^\infty(M)$.

By Lemma~\ref{lem: C GH closed} the subset
$\O^\infty_{\ge \l}(M)$ is closed in the  Fr\'echet space $C^\infty(M)$.
Like any closed convex non-locally-compact
subset of a  Fr\'echet space, 
$\O^\infty_{\ge \l}(M)$ is homeomorphic to $\ell^2$, see~\cite[Theorem 2]{DobTor}. 
Similarly, $D^\infty(M)$ is a contractible Fr\'echet manifold, hence it
is homeomorphic to $\ell^2$, see~\cite{EarSch, Yag},
and the Fr{\'e}chet space $\ell^2\times\ell^2$ is isomorphic to $\ell^2$. 
Thus $\mathcal R_{\ge \l}^\infty(M)$ is homeomorphic to $\ell^2$, where as usual
$\l=0$ for $M=\C$.

If $M$ is $S^2$ or $RP^2$, then $\O_{>\l}^\infty(M)$ is an open contractible subset
in the  Fr\'echet space $C^\infty(M)$, and hence it is homeomorphic to $\ell^2$.
It remains to prove the following.

\begin{thm}
\label{lem: O>0(C) is l2}
$\O_{>0}^\infty(\C)$ is homeomorphic to $\ell^2$. 
\end{thm}
\begin{proof} 
By Toru\'nczyk's characterization theorem, see~\cite[Theorem 1.1.14]{BRZ-book}, 
a Polish AR is homeomorphic to $\ell^2$
if and only if it has the Strong Discrete Approximation Property (SDAP).
Also a space is an AR with SDAP if and only if it is homeomorphic to a homotopy dense subset of
$\ell^2$~\cite[Theorem 1.3.2]{BRZ-book}.
Note that $\O^\infty_{>0}(\C)$ is dense
in $\O^\infty_{\ge 0}(\C)$: Any $u\in\O^\infty_{\ge 0}(\C)$ can be approximated by the convex combinations
$(1-t)u+tu_0$ which lie in $\O^\infty_{>0}(\C)$ for $t\in (0,1)$ provided $u_0\in \O_{>0}^\infty(\C)$.
If a convex subset of a linear metric space is dense in an AR, then it is homotopy dense
in that AR~\cite[exercises 12-13 in 1.2]{BRZ-book}. Thus $\O^\infty_{>0}(\C)$ is homotopy dense in 
$\O^\infty_{\ge 0}(\C)$.
Since $\O^\infty_{\ge 0}(\C)$ is homeomorphic to
$\ell^2$, it remains to check that $\O^\infty_{>0}(\C)$ is Polish.
Recall that a subspace of a Polish space is Polish if and only if it is $G_\delta$. 
To show that $\O^\infty_{>0}(\C)$ is $G_\delta$ note that\[
\O^\infty_{>0}(\C)=\bigcap_n\,\{u\in \O^\infty_{\ge 0}(\C)\co \triangle_{g_{_0}} u
\left\vert_{\{z\co |z|\le n\}}>0\}\right.,
\]
which is a countable intersection of open sets.
\end{proof}

\section{$\M_2$-absorbing diffeomorphism groups}
\label{sec: abs diff gr}

In this section $\b\in [1,\infty)$ and $M$ is $S^2$, $RP^2$ or $\C$.
The group $D^\b(M)$ was defined in Section~\ref{sec: unif M} for compact $M$
and in Section~\ref{sec: unif of C} for $M=\C$.
The results of this section combine to prove the following theorem.

\begin{thm}
$D^\b(M)$ is $\M_2$-absorbing for each $\b\in [1,\infty)$. 
\end{thm}

Recall that any contractible $\M_2$-absorbing space is homeomorphic to $\Si^\o$.
We noted in Section~\ref{sec: unif M} that $D^\b(S^2)$ and $D^\b(RP^2)$ are contractible for $\b\notin \Z$,
and $D^\b(S^2)$ is contractible for every integer $\b>1$.

\begin{lem} 
\label{lem: diff C is contractible}
$D^\b(\C)$ is contractible. 
\end{lem}
\begin{proof}
Denote by $\Diff^+(\C)$ the group of all smooth orientation-preserving
diffeomorphisms of $\C$; we equip its subgroups with the $C^{\b}$ topology.
Let $\Conf^+(\C)$ denote the subgroup of conformal automorphisms of $\C$, and 
let $D^\b_*(\C)$ be the subgroup of diffeomorphisms 
of $\C$ that fix $0$ and whose differential at $0$ is the identity.
Again $D^{\b}_*(\C)$ is a topological group.
The subgroups $D^{\b}_*(\C)$ and $D^{\b}(\C)$ are closed in 
$\Diff^+(\C)$ and intersects $\Conf^+(\C)$ 
trivially. Any element of $\Diff^+(\C)$ can be written uniquely
as the product of an element of $\Conf^+(\C)$ with
an element of either subgroup, which defines a homeomorphism
of  $\Diff^+(\C)$ to the product of $\Conf^+(\C)$ with either subgroup.  
The group $D^\b_*(\C)$ are contractible via the homotopy $H_t(f)(v)=\frac{f(tv)}{t}$ for $t\in (0,1]$
and $H_0(f)(v)=f_*(0)(v)$. Hence the inclusion $\Conf^+(\C)\to \Diff^+(\C)$
is a homotopy equivalence, which implies 
the contractibility of $D^\b(\C)$. Indeed,
the slice inclusion $D^{\b}(\C)\to D^{\b}(\C)\times \Conf^+(\C)$ followed
by the projection to the first factor is the identity,
while the composition of the slice inclusion with the projection
on the second factor, which is a homotopy equivalence, is a constant map.
\end{proof}

\begin{lem}
$D^{\b}(M)$ is in $\M_2$.
\end{lem} 
\begin{proof}
Since $D^{\infty}(M)$ is homeomorphic to $\ell^2$, 
see Corollary~\ref{cor: diff loc homeo to normed space} for compact $M$
and~\cite[Theorem 1.1]{Yag} for $M=\C$, the space  $D^{\infty}(M)$
is completely metrizable, and hence Lemma~\ref{lem: O is in M2} 
implies the claim.
\end{proof} 

\begin{lem}
$D^{\b}(M)$ is  $\s Z$.
\end{lem} 
\begin{proof}
First assume that $M$ is $S^2$ or $RP^2$.  Corollary~\ref{cor: diff loc homeo to normed space}
shows that $D^{\b}(M)$ is locally homeomorphic to the linear space $V^\b$ of smooth vector fields
on $M$ equipped with the $C^{\b}$ topology. Hence $V^\b$ is the image of 
the inclusion $\iota$ from the Fr{\'e}chet space $V^\infty$ of smooth vector fields on $M$ to the Banach space
of $C^\b$ vector fields on $M$. Since $\b$ is finite, $V^\b$ is incomplete. 
As noted in~\cite[Proposition 3.6]{BDP} any incomplete operator image (i.e. the image of
a continuous linear operator between Fr{\'e}chet spaces) is $\s Z$. More precisely,
there is an open convex neighborhood $U$ of $0$ in $V^\infty$ such that closure $\overline{\iota(U)}$
of $\iota(U)$ in $V^\b$ is a $Z$-set and so is $n\cdot\overline{\iota(U)}$ for every $n\in \N$.
Thus $V^\b$ is a countable union of $Z$-sets $n\cdot\overline{\iota (U)}$. 
Moreover, one can choose
$U$ so that $\overline{\iota(U)}$ is bounded which implies that the homeomorphism of $V^\b$ onto 
an open subset of $D^\b(M)$ takes each $n\cdot\overline{\iota (U)}$ to a closed subset. 
Since $D^\b(M)$ is separable, it is covered by countably many open sets homeomorphic to $V^\b$,
hence $D^\b(M)$ is $\s Z$.

Let $M=\C$, and let $D$ be a closed disk embedded in $\C$.
Fix an integer $n>\b$, set $r=n+3$, and
let $Z_m$ be the subspace of all $z\in D^{\b}(\C)$ such that 
for every $\r\in (\b, r)$ the $C^\r$ norm of 
$z\vert_{_{D}}$ is at most $m$. 

As in Lemma~\ref{lem: O is sigmaZ} below we see that $D^{\b}(\C)=\bigcup_{m>0} Z_m$,
and each $Z_m$ is closed in $D^{\b}(\C)$.
To show that $Z_m$ is a $Z$-set it remains to check that 
any continuous map $f$ from the Hilbert cube $Q$ into $D^{\b}(\C)$
is a uniform limit of maps whose ranges miss $Z_m$.

For $X\in\{\C, \overline\C\}$ let $C^{\b}_D(X,X)$ be the subspace of $C^{\b}(X,X)$
consisting of maps that are the identity outside $D$. Here $\overline\C=\C\cup\{\infty\}$
is the Riemann sphere. 
The obvious extension map $C^{\b}_D(\C,\C)\to C^{\b}_D(\overline\C,\overline\C)$ is a homeomorphism.
Given $h\in C^{\b}_D(\C,\C)$ we denote its extension to $\overline\C$ by $\overline h$. 
Clearly $\overline h$ is a $C^\b$ diffeomorphism if and only if so is $h$. 

For a function $p\in C^\infty(\C)$ with support in $D$ 
we set $f_p(q,t)=tp+(1-t)f(q)$ where $q\in Q$ and $t\in J=[0,1)$.
The associated map $f_p\co Q\times J\to C^{\b}(\C,\C)$ is clearly continuous.
We are going to show that for a suitable $p$ and small enough $t$ the map $f_p(\cdot, t)$
takes values in $D^\b(\C)\setminus Z_m$ and approximates $f$.

The map $f_0(q,t)=(1-t)f(q)$ lies in $D^\infty(\C)$ as the composite of $f(q)$ and the scaling by $1-t$.
Set $h_p(q,t)=f_0(q,t)^{-1}\circ f_p(q,t)$; note that $h_p(q,t)$ is
the identity outside $D$ for every $q$, $t$. Thus each $h_p$ gives rises to a continuous map 
$\overline h_p\co Q\times J\to C^{\b}_D(\overline\C,\overline\C)$.

Fix a real-valued function $\a\in C^{n}(\C)$ such that $\a$ is not $C^{n+1}$ and 
the support of $\a$ is in $D$. Then the map $f_{\a}\co Q\times I\to C^{\b}(\C,\C)$ is continuous, 
and $f_\a(q,t)$ is not $C^{n+1}$ for each $t>0$.
Note that $\overline h_\a(q,0)$ is the identity for every $q$. Recall that 
\begin{itemize}
\item[(i)]
if $L$ is a compact manifold without boundary, then
$C^1$ diffeomorphisms of $L$ form an open subset in $C^1(L,L)$~\cite[Theorem 1.7 in Chapter 2]{Hir};
\vspace{3pt}
\item[(ii)] 
a $C^l$ map with $l\ge 1$ is a $C^l$-diffeomorphism if and only if it is a $C^1$-diffeomorphism
(by the inverse function theorem).
\end{itemize}
Since $n>\b\ge 1$ and $Q$, $\overline \C$  are compact, (i)--(ii) imply that there is an $\e>0$ such that  
$\overline h_{\a}(\cdot,t)$ takes values in $D^{n}(\overline\C)$ for all $t\in (0,2\e)$. 

Next let us approximate $\a$ in the $C^n$ topology by a sequence of functions $\a_i\in C^\infty(\C)$ 
with supports in $D$.
Since $\overline h_{\a}(q,\e)\in D^{n}(\overline\C)$ and $\overline h_{\a_i}(q,\e)$ is $C^\infty$, (i)--(ii) again
imply that $\overline h_{\a_i}(q,\e)\in D^{\infty}(\overline\C)$ for all large $i$.
Restricting $\overline h_{\a_i}(q,\e)$ to $\C$ and post-composing the restriction with $f_0(q,t)$
gives $f_{\a_i}(q,\e)\in D^\infty(\C)$ for all large $i$.

By the Arzel{\`a}-Ascoli theorem the $C^{n+2}$ norm of $f_{\a_i}(q,\e)\vert_{_{D}}$ tends to infinity uniformly in $q$
(else some sequence $f_{\a_i}(q_i,\e)\vert_{_{D}}$ would converge to a $C^{n+1}$ map
of the form $f_{\a}(q,\e)\vert_{_{D}}$ contradicting the choice of $\a$). In particular,
the range of $f_{\a_i}(\cdot,\e)$ misses $Z_m$ for all large $i$.
Finally, $f_{\a_i}(\cdot,\e)$ is an approximation of $f$ because $f-f_{\a_i}(\cdot,\e)=\e(f-\a_i)$
and the $C^\b$ norm of $f-\a_i$ is uniformly bounded as $\a_i$ converges to $\a$ in the $C^n$ topology and
$n>\b$.
\end{proof}

\begin{lem}
\label{lem: diff C is ANR}
$D^{\b}(M)$ is homeomorphic to
a homotopy dense subset of $\ell^2$, and in particular it is an ANR.
\end{lem}
\begin{proof}
If $M$ is $S^2$ or $RP^2$, then Corollary~\ref{cor: diff loc homeo to normed space}
implies that $D^{\b}(M)$ is an ANR.
The fact that $D^{\b}(\C)$ is an ANR is proved
using a method of T.~Yagasaki~\cite{Yag}. To conform 
to notations in~\cite{Yag} denote $D^{\b}(\C)$ by $G$,
and let $G_K$ denote the subgroup fixing a subset $K$ pointwise.
Fix an exhaustion of $\C$ by closed round disks $M_i$ centered at $0$ of radius $i>0$.
Then $G_{M_i}$ is contractible by the Alexander trick towards infinity.
Now~\cite[Theorem 3.1(2)(i)]{Yag} gives a sufficient condition, 
called Assumption (A) in~\cite[p.8]{Yag}, under which 
the identity component of $G$ is an ANR.
Since $G$ is path-connected by Lemma~\ref{lem: diff C is contractible}, 
it remains to check Assumption (A)
which is somewhat technical to state but the two main points  
are  as follows. The group $G_{\C\setminus M_i}$
is an ANR because it is locally homeomorphic to a normed space
by the proof of Corollary~\ref{cor: diff loc homeo to normed space}. 
The (one paragraph) proof in~\cite{Lim} of the parametrized isotopy extension theorem
extends to the setting of smooth maps with H\"older $C^{\b}$ topology.

Finally, by~\cite[Theorem 4.2.1 and Theorem 1.3.2]{BRZ-book} any infinite-dimensional
ANR-group is homeomorphic to
a homotopy dense subset of an $\ell^2$-manifold,
which in our case is $\ell^2$ because $D^{\b}(M)$ 
is contractible.
\end{proof}

\begin{lem}
$D^{\b}(M)$ is strongly $\M_2$-universal.
\end{lem} 
\begin{proof}
As usual we plan to establish $\M_2$-universality by 
applying Theorem~\ref{thm: criterion M_2-univer} in which we let $N=[0,1]$,
$h_\bu(x)=x$, $\eta=\frac{1}{2}$, $D=\left[\frac{1}{2}, \frac{2}{3}\right]$, 
and $\mathfrak D u=u^\prime$.
For $n\ge 1$ let $C^n_\bu$ be the subspace of $C^n([0,1])$ consisting of functions
$u$ that equal $h_\bu$ outside $D$ and satisfy 
$u^\prime\ge\eta$ on $D$. By Theorem~\ref{thm: criterion M_2-univer}
the assumptions of Corollary~\ref{cor: Cbullet is M2-universal}
are satisfied. To apply the corollary fix $n>\b$ and construct a map 
$f\co C^n_\bu\to D^{\b}(M)$ as follows. Fix an closed coordinate disk $B$
in $M$, identify it with a unit disk in $\C$, and think of it in 
polar coordinates. Let $f(u)$ be the diffeomorphism  of $M$
that is the identity outside $B$, while on $B$ it takes
$re^{i\th}$ to $u(r)e^{i\th}$. 
Clearly $f$ is injective and its continuity follows from $n>\b$.
It is easy to check that the image of $f$ is closed.
Thus $D^{\b}(M)$ is $\M_2$-universal.
Clearly $D^{\b}(M)$ is infinite dimensional, and it is an ANR by
Lemma~\ref{lem: diff C is ANR}. Hence~\cite[Theorem 4.2.3]{BRZ-book} implies
strong $\M_2$-universality of $D^{\b}(M)$.
\end{proof}

\section{The spaces $\O^\g_{\ge\l}(M)$ and $\O^\g_{>\l}(M)$ are $\M_2$-absorbing}
\label{sec: O absorb}

Let $\g\in [0,\infty)$ and $M$ be $S^2$, $RP^2$ or $\C$ 
equipped with a complete metric $g_\k$ of constant curvature $\k\in\{0,1\}$.

In this section we adopt the following convention: 
whenever we write $\O^\g_{\ge\l}(\C)$ and $\O^\g_{>\l}(\C)$ we assume $\l=0$.
This will allow us to treat all surfaces at once.

\begin{lem} 
\label{lem: O is sigmaZ} 
If $\g$ is finite, then the following spaces are $\s Z$
\begin{itemize}
\item[(1)]
$\O^\g_{\ge\l}(S^2)$, $\O^\g_{>\l}(S^2)$, $\O^\g_{\ge\l}(RP^2)$, $\O^\g_{>\l}(RP^2)$ 
for any $\l\in\R$,\vspace{3pt}
\item[(2)]
$\O^\g_{\ge 0}(\C)$ and $\O^\g_{>0}(\C)$. 
\end{itemize}
\end{lem}
\begin{proof} 
We first give a proof for $\O^\g_{>\l}(M)$ and then indicate necessary modifications for 
$\O^\g_{\ge\l}(M)$.
Fix $r>\max\{\g, 2\}$, and let $D$ be a closed disk embedded in $M$. 
For a positive integer $m$ let $Z_m$ be the set of all functions $u\in \O^\g_{>\l}(M)$
such that $\|u\vert_{_D}\|_{_{C^\r(D)}}\le m$ for every $\r\in (\max\{\g,2\}, r)$.

The equality $\O^\g_{>\l}(M)=\bigcup_m Z_m$ follows from the facts that
any smooth function on $D$ has finite $C^r$ norm, and 
the Lipschitz constant of the identity map of $C^r(D)$,
where the domain and the co-domain are given the $C^r$ and $C^\r$ norms respectively,
is bounded above independently of $\r$, see~\cite[Lemma 6.35]{GilTru}. 

To show that $Z_m$ is closed consider $u_i\to u$ in $\O^\g_{>\l}(M)$ with $u_i\in Z_m$, and
fix $\r<\de<r$. Since $\{u_i\vert_{_D}\}$ is uniformly bounded in the $C^\de$ norm, 
$u_i\vert_{_D}$ subconverges in the $C^\r$ topology, see~\cite[Lemma 6.36]{GilTru}. 
The limit equals $u\vert_{_D}$
because $\r>\g$, and hence $\|u\vert_{_D}\|_{_{C^\r(D)}}\le m$.

To show that $Z_m$ is a $Z$-set we fix a continuous 
map $f\co Q\to \O^\g_{>\l}(M)$, where $Q$ is the Hilbert cube, 
and approximate it by a map whose range misses $Z_m$. 

For the first step of the approximation
consider a partition of unity $\mathfrak u=\{\psi_i, U_i\}$, $1\le i\le k$ on $Q$ where $\psi_i$
is a continuous real-valued function on $Q$ with support in $U_i$. 
For each $i$ pick a point $q_i\in U_i$, and consider the function $f_{\mathfrak u}=\sum_{i=1}^k f(q_i)\psi_i$.
Thus $f_{\mathfrak u}$ takes values in the convex hull
of the functions $f(q_1),\dots, f(q_k)$, which is a finite dimensional convex subset of $\O^\g_{>\l}(M)$.
Let $\{p_j\}$ be a countable family of seminorms defining the topology on
the ambient vector space $\G^\g(M)$ with the associated metric $d=\sum_{j\ge 1}\frac{p_j}{p_j+1}2^{-j}$.
For any $\e>0$ there is a partition of unity as above such that for any $i$, $j$, and
$x,y\in f(U_i)$ we have $p_j(x,y)<\e$, and therefore $\displaystyle{\sup_{q\in Q} d(f(q),f_{\mathfrak u}(q))<\e}$.

For the second stage of the approximation
let $f_{\mathfrak u,\a}(q)=f_\mathfrak u(q)+\a$ where $\a$ is a real valued smooth function 
on $M$ supported in $D$ and such that 
$\|\a\vert_{_D}\|_{_{C^{\g}(D)}}$ and $\|\a\vert_{_D}\|_{_{C^{2}(D)}}$ are small.
If $M=\C$, then for each $q$ the functions $f_{{\mathfrak u},\a}(q)$ and $f_{\mathfrak u}(q)$ have the same growth at infinity
as $\a$ is compactly supported. 
Compactness of $Q$ and strictness of the inequality ``$>\l$'' makes
the image of $f_{{\mathfrak u},\a}$ lie in $\O^\g_{>\l}(M)$ provided
$\|\a\vert_{_D}\|_{_{C^{2}(D)}}$ is sufficiently small.

The map $q\to f_{\mathfrak u}(q)\vert_{_D}$ takes values in a finite dimensional vector subspace 
of $\G^\g(D)$, and is continuous with respect to the $C^\g$ norm on the co-domain,
and hence with respect to any norm, as they are all equivalent. In particular, compactness
of $Q$ implies that there is $R>0$ with $\|f_{\mathfrak u}(q)\vert_{_D}\|_{_{C^\r(D)}}\le R$
for each $q\in Q$ and every $\r\in (\max\{\g,2\},r)$.

We can also choose $\a$ so that $\|\a\vert_{_D}\|_{_{C^{\s}(D)}}>m+R$
for some $\s\in (\max\{\g,2\}, r)$. 
Then the image of $f_{\mathfrak u,\a}$ is disjoint from $Z_m$ for 
if $f_{\mathfrak u,\a}(q)\in Z_m$ for some $q$, then
\[
\|\a\vert_{_D}\|_{_{C^\s(D)}}=
\|f_{{\mathfrak u},\a}(q)\vert_{_D}-f_{\mathfrak u}(q)\vert_{_D}\|_{_{C^\s(D)}}\le m+R.
\]

To treat the case of $\O^\g_{\ge \l}(M)$ we need one more step in the approximation.
Fix any function $v\in\O^\g_{>\l}(M)$.
(If $M$ is compact, we can let $v$ be a sufficiently large constant
such that $e^{2v}>\l$, and if $M=\C$, then take $v$ be any function such that
$e^{-2v}g_{_{0}}$ has a complete metric of positive curvature).
Now any continuous map $f\co Q\to \O^\g_{\ge \l}(M)$
can be approximated by the map $q\to (1-t)f(q)+tv$ for small positive $t$,
and the image of the latter map lies in $\O^\g_{>\l}(M)$ 
(see the proof of Lemma~\ref{lem: convexity of exp} and note that the leftmost
inequality in the displayed formula applied to $(1-t)f(q)+tv$ 
is strict for $t\in (0,1)$ because $v\in\O^\g_{>\l}(M)$).
\end{proof}

\begin{lem} 
\label{lem: O is M_2}
Each space in the statement of Lemma~\ref{lem: O is sigmaZ} 
belongs to $\M_2$. 
\end{lem}
\begin{proof}
By Lemma~\ref{lem: O is in M2} it suffices to show that the corresponding
spaces with $\g=\infty$ are completely metrizable.
A subspace of a complete metric space is completely metrizable
if and only if it is $G_\de$.
The latter is true for $\O^\infty_{>0}(\C)$ because it is $G_\de$
in the ambient Fr{\'e}chet spaces by
the proof of Theorem~\ref{lem: O>0(C) is l2}. If $M$ is $S^2$ or $RP^2$, then 
$\O^\infty_{\ge\l}(M)$ and $\O^\infty_{>\l}(M)$ are respectively
closed and open, and hence are $G_\de$.
Finally, $\O^\infty_{\ge 0}(\C)$ is closed by~\cite[Lemma 2.4]{BelHu-modr2}.
\end{proof}

\begin{lem}
\label{lem: O(M) is in M2-universal}
Each space in the statement of  Lemma~\ref{lem: O is sigmaZ} is $\M_2$-universal.
\end{lem}
\begin{proof} 
Recall that a space is $\M_2$-universal if it contains an $\M_2$-universal
closed subset, and to find one we are going to apply Corollary~\ref{cor: Cbullet is M2-universal} 
and Theorem~\ref{thm: criterion M_2-univer}. To satisfy the assumptions of 
Theorem~\ref{thm: criterion M_2-univer}
let $N=M$, $l=2$, $\mathfrak D u=e^{2u}(\k+\triangle_{g_\k})$,
fix a function $h_\bu\in\O^\g_{>\l}(M)$, fix a closed coordinate disk $D$ in $M$, and 
let $\eta$ be a constant such that $\max\{0,\l\}<\eta<\min\mathfrak D h_\bu\vert_D$.
With these data for $n\ge\max\{l,\g\}$ define $C^{n}_\bu$ as in Theorem~\ref{thm: criterion M_2-univer}, i.e.,
let $C^n_\bu$ denote the subspace of $C^n(M)$ of functions $u$ 
such that $\mathfrak D u\vert_D\ge \eta$ and $u\vert_{M\setminus \mathrm{Int}\,D}=h_\bu$. 
Corollary~\ref{cor: Cbullet is M2-universal} applied to the inclusion
of $C^n_\bu$ into $\G^\g(M)$ shows that $C_\bu^\infty=\bigcap_n C^n_\bu$
is $\M_2$-universal. Using the proof of Lemma~\ref{lem: C GH closed} we see that
that $C_\bu^\infty=\bigcap_n C^n_\bu$ is closed in $\G^\g(M)$, and
by construction $C_\bu^\infty$ lies on $\O^\g_{>\l}(M)$.
\end{proof}

\begin{thm}
\label{thm: O is absorbing}
Each space in the statement of  Lemma~\ref{lem: O is sigmaZ}
is $\M_2$-absorbing, and in particular, is homeomorphic to $\Si^\o$.  
\end{thm}
\begin{proof}
Let $\overline{\G}^{\,\g\!}(M)$ be the closure of 
$\G^\g(M)$ in the Fr{\'e}chet 
space of all $C^\g$ functions on $M$ with the $C^\g$ topology;
thus $\overline{\G}^{\,\g\!}(M)$ is a separable Fr{\'e}chet space. 
The closure of $\O^\g_{>\l}(M)$ in $\overline{\G}^{\,\g\!}(M)$ 
is not contained in $\mathrm{Aff}\, \O^\g_{>\l}(M)$
because it clearly contains some non-smooth rotationally symmetric functions;
the same therefore holds for $\O^\g_{\ge\l}(M)$.
Hence by Lemma~\ref{lem: convex embeds as homot dense} $\O^\g_{>\l}(M)$, $\O^\g_{\ge\l}(M)$ 
are homeomorphic to homotopy dense subsets of $\ell^2$. The proof of 
Lemma~\ref{lem: O(M) is in M2-universal} actually shows that 
the $\M_2$-universal subset $f(C^\infty_\bu)$
is closed in $\G^\g(M)$. Therefore \cite[Proposition 5.3.5]{BRZ-book}
implies that $\O^\g_{>\l}(M)$, $\O^\g_{\ge\l}(M)$  are strongly $\M_2$-universal.
Thus Lemmas~\ref{lem: O is sigmaZ}, \ref{lem: O is M_2}, 
imply that $\O^\g_{>\l}(M)$, $\O^\g_{\ge\l}(M)$ are $\M_2$-absorbing.
\end{proof}

\small
\bibliographystyle{amsalpha}
\bibliography{met2d-revised-14Oct2016}

\end{document}